\documentclass[11pt, oneside]{article}   	% use "amsart" instead of "article" for AMSLaTeX format
\usepackage{geometry}                		% See geometry.pdf to learn the layout options. There are lots.
\usepackage{graphicx}				% Use pdf, png, jpg, or eps§ with pdflatex; use eps in DVI mode
\usepackage{mathrsfs,amssymb,amsmath,amsthm, color,stmaryrd}

\makeatletter
\def\thanks#1{%
   \footnotemark
   \edef\@tempa{\noexpand\noexpand\noexpand\footnotetext[\the\c@footnote]}%
   \toks@\expandafter{\@thanks}%
   \toks\tw@{{#1}}
   \xdef\@thanks{\the\toks@\@tempa\the\toks\tw@}}
\makeatother

\newtheorem{thm}{Theorem}[section]

\newtheorem{defi}[thm]{Definition}
\newtheorem{lem}[thm]{Lemma}
\newtheorem{prop}[thm]{Proposition}
\newtheorem{cor}[thm]{Corollary}

\newtheorem{rem}[thm]{Remark}
\numberwithin{equation}{section}

\newcommand{\bbN}{\mathbb{N}}
\newcommand{\bbE}{\mathbb{E}}
\newcommand{\bbP}{\mathbb{P}}
\newcommand{\bbT}{\mathbb{T}}
\newcommand{\bbR}{\mathbb{R}}
\newcommand{\bbZ}{\mathbb{Z}}

\newcommand{\mcA}{\mathcal{A}}

\newcommand{\mcD}{\mathcal{D}}

\newcommand{\mcF}{\mathcal{F}}
\newcommand{\mcI}{\mathcal{I}}
\newcommand{\mcJ}{\mathcal{J}}
\newcommand{\mcK}{\mathcal{K}}
\newcommand{\mcL}{\mathcal{L}}

\newcommand{\mcN}{\mathcal{N}}
\newcommand{\mcP}{\mathcal{P}}

\newcommand{\mcR}{\mathcal{R}}

\newcommand{\scM}{\mathscr{M}}
\newcommand{\scT}{\mathscr{T}}

\newcommand{\bfA}{\mathbf{A}}

\newcommand{\bfT}{\mathbf{T}}
\newcommand{\bfG}{\mathbf{G}}
\newcommand{\bfS}{\mathbf{S}}

\newcommand{\wco}{w^*}

\newcommand{\mfs}{\mathfrak{s}}

\newcommand{\id}{\mathop{\text{\rm id}}}

\newcommand{\lp}{\llparenthesis\,}
\newcommand{\rp}{\,\rrparenthesis}

\newcommand{\lb}{\llbracket}
\newcommand{\rb}{\rrbracket}

\newcommand{\tri}{|\!|\!|}

\newcommand{\spa}{\mathop{\text{\rm span}}}
\newcommand{\per}{\text{\rm ad}}

\allowdisplaybreaks

\title{A semigroup approach to the reconstruction theorem and the multilevel Schauder estimate for singular modelled distributions}
\author{Masato Hoshino\thanks{
Graduate School of Engineering Science, Osaka University, 
1-3, Machikaneyama, Toyonaka, Osaka, 560-8531, Japan. 
Email: {\tt hoshino@sigmath.es.osaka-u.ac.jp}
}
and Ryoji Takano\thanks{
Graduate School of Engineering Science, Osaka University, 
1-3, Machikaneyama, Toyonaka, Osaka, 560-8531, Japan. 
Email: {\tt rtakano@sigmath.es.osaka-u.ac.jp}
}}

\date{}							% Activate to display a given date or no date

\begin{document}
\maketitle
%\section{}
%\subsection{}

\begin{abstract}
We extend the semigroup approach used in \cite{OW19, Ho23} to provide alternative proofs of the reconstruction theorem and the multilevel Schauder estimate for singular modelled distributions. 
As an application of them, we construct the local-in-time solution of the two dimensional parabolic Anderson model with a non-translation invariant differential operator.
\end{abstract}

%%%%%%%%%%%%%%%%%%%%%%%%%%%%%%%%%%%%
\section{Introduction}\label{sec:intro}
%%%%%%%%%%%%%%%%%%%%%%%%%%%%%%%%%%%%

The theory of regularity structures established by Hairer \cite{Hai14} provides a robust framework adapted to a wide class of (subcritical) singular stochastic PDEs.
One of the most important concepts in this theory is the notion of \emph{modelled distributions}, which are considered as ``generalized Taylor expansions" of the solutions to the underlying equations.
The analytic core of the theory is to prove two key theorems for modelled distributions: the \emph{reconstruction theorem} \cite[Theorem 3.10]{Hai14} and the \emph{multilevel Schauder estimate} \cite[Theorem 5.12]{Hai14}.
The former theorem constructs a global distribution by gluing local distributions derived from a given modelled distribution together.
The latter translates an integral operator such as the convolution operator with Green function into the operator on the space of modelled distributions.
Since Hairer first proved the reconstruction theorem, some alternative proofs have been proposed using various approaches, such as Littlewood--Paley theory \cite{GIP15}, the heat semigroup approach \cite{OW19, BH20}, the mollification approach \cite{ST18}, and the convolution approach \cite{FH20}.
Inspired by \cite{OW19}, the first author of this paper proved both theorems by using the operator semigroup in \cite{Ho23}.
On the other hand, Caravenna and Zambotti \cite{CZ20} introduced the notion of \emph{germs} to describe the analytic core of the proof of the reconstruction theorem, and later, they and Broux \cite{BCZ23} proved the multilevel Schauder estimate at the level of germs.
See also \cite{HL17, DDD19, HR20, LPT21, RS21, BL22, ZK23,  HS23+} for extensions of the theorems into different settings, such as Besov or Triebel--Lizorkin norms, or Riemannian manifolds.
See also \cite{FS22} for a Besov extension of the sewing lemma, which plays a role similar to the reconstruction theorem in rough path theory.

In the aforementioned literatures, modelled distributions are often defined on the entire space $\bbR^d$ to avoid technical difficulties related to boundary conditions.
However, it is not sufficient for applications. 
To apply the theory of regularity structures to parabolic equations, it is necessary to define modelled distributions on the time-space region $(0,\infty)\times\bbR^d$ allowing a singularity at the hyperplane $\{0\}\times\bbR^d$.
This modified version of modelled distributions is called \emph{singular modelled distributions}.
In \cite[Section 6]{Hai14}, the reconstruction theorem and the multilevel Schauder estimate were extended to the class of singular modelled distributions.
An extension to Besov norms is demonstrated in \cite{HL18}, and boundary conditions on both time and space variables are considered in \cite{GH19}.
However, compared to the case of modelled distributions without boundary conditions, there seems to be a less number of studies on alternative proofs and extensions.
It should be mentioned that, in the context of rough path theory, the sewing lemma is extended into the singular path spaces allowing a singularity at time $t=0$ by \cite{BFG}.

The aim of this paper is to extend the semigroup approach used in \cite{Ho23} and provide alternative proofs of the reconstruction theorem (see Corollary \ref{corollary:singularreconstruction}) and the multilevel Schauder estimate (see Corollary \ref{corollary:singularMSE}) for singular modelled distributions.
The proofs use arguments similar to \cite{Ho23}, but require the following technical modifications.
\begin{itemize}
\item[(i)]
Following \cite{Ho23}, we define Besov norms using the operator semigroup $\{Q_t\}_{t>0}$.
The associated integral kernel $Q_t(x,y)$ is inhomogeneous and has restricted regularities with respect to $x$ and $y$ in general.
Hence the equivalence between the norm associated with $\{Q_t\}_{t>0}$ and the standard norm defined from Littlewood--Paley theory is uncertain.
For this reason, we need some nontrivial arguments to prove the uniqueness of the reconstruction.
\item[(ii)]
Since $Q_t$ is an integral operator defined over the entire spacetime, we always require global bounds on models and modelled distributions, unlike the original definitions in \cite{Hai14} that assume only local bounds.
Consequently, in addition to the definition of singular modelled distributions (see Definition \ref{def:besovsMD}) which is closer to the original one, we use a different definition that assumes global bounds (see Proposition \ref{def:variantsnorms}-\ref{def:variantsnorms1}).
For this reason, as for the existence of the reconstruction, we assume a stronger condition ``$\eta-\gamma>-\mfs_1$" for the parameters appearing in the definition of singular modelled distributions than the condition ``$\eta>-\mfs_1$" as in \cite{Hai14}. 
It is not actually a serious problem in applications because we can switch to a small $\gamma$ to apply the reconstruction theorem.
\end{itemize}

Moreover, as an application, we discuss the parabolic Anderson model (PAM)
$$
\big(\partial_1-a(x)\Delta\big)u(t,x)=b\big(u(t,x)\big)\xi(x)
\qquad((t,x)\in(0,\infty)\times\bbT^2)
$$
with a spatial white noise $\xi$.
Here $b:\bbR\to\bbR$ is in the class $C_b^3$ and $a:\bbT^2\to\bbR$ is an $\alpha$-H\"older continuous function for some $\alpha\in(0,1)$ and satisfies
$$
C_1\le a(x)\le C_2\qquad(x\in\bbT^2)
$$
for some constants $0<C_1<C_2$.
When $a$ is a constant, the above equation is one of the simplest examples of subcritical singular stochastic PDEs, as studied in \cite{Hai14, DDD19}.
We show that the equation with general coefficients as above can be renormalized, with the spacetime dependent renormalization function (see Theorem \ref{5:theorem:renorPAM}).
Such ``non-translation invariant" equations are more generally studied by \cite{BB21, Sin23}.
The aim of this paper is to deepen the analytic core of \cite{BB21}, which uses the semigroup approach.
On the other hand, \cite{Sin23} is a direct extension of \cite{Hai14}.
One of the differences between this paper and \cite{Sin23} is in the requirements of the smoothness of coefficients.
In \cite{Sin23}, a bit of smoothness is required, but in this paper the coefficients only need to have positive H\"older continuities.

This paper is organized as follows.
In Section \ref{section:preliminary}, we recall from \cite{Ho23} Besov norms associated with the operator semigroup, and prove important inequalities used throughout this paper.
In Section \ref{section:singular}, we recall the basics of regularity structures and prove the reconstruction theorem for singular modelled distributions.
Section \ref{section:schauder} is devoted to the proof of the multilevel Schauder estimate for singular modelled distributions.
In Section \ref{section:application}, we discuss an application to the two-dimensional PAM.

\

%%%%%%%%%%%%%%%%%%%%%%%%%%%%%%%%%%%%
{\bf \noindent Notations}
%%%%%%%%%%%%%%%%%%%%%%%%%%%%%%%%%%%%

\

The symbol $\bbN$ denotes the set of all nonnegative integers.
Until Section \ref{section:schauder}, we fix an integer $d\ge1$, the \emph{scaling} $\mfs=(\mfs_1,\dots,\mfs_d)\in[1,\infty)^d$, and a number $\ell>0$. We define $|\mfs|=\sum_{i=1}^d\mfs_i$.
For any multiindex ${\bf k}=(k_i)_{i=1}^d\in\bbN^d$, any $x=(x_i)_{i=1}^d\in\bbR^d$, and any $t>0$, we use the following notations.
\begin{align*}
&{\bf k}!:=\prod_{i=1}^dk_i!,\quad
|{\bf k}|_{\mfs}:=\sum_{i=1}^d\mfs_ik_i,\quad
\|x\|_\mfs:=\sum_{i=1}^d|x_i|^{1/\mfs_i},\\
&x^{\bf k}:=\prod_{i=1}^dx_i^{k_i},\quad
t^{\mfs/\ell}x:=(t^{\mfs_i/\ell}x_i)_{i=1}^d,\quad
t^{-\mfs/\ell}x:=(t^{-\mfs_i/\ell}x_i)_{i=1}^d.
\end{align*}
We define the set $\bbN[\mfs]:=\{|{\bf k}|_\mfs\,;\,{\bf k}\in\bbN^d\}$, which will be used in Section \ref{section:schauder}. The parameter $t$ is not a physical time variable, but an auxiliary variable used to define regularities of distributions.
For multiindices ${\bf k}=(k_i)_{i=1}^d$ and ${\bf l}=(l_i)_{i=1}^d$, we write ${\bf l}\le{\bf k}$ if $l_i\le k_i$ for any $1\le i\le d$, and then define $\binom{{\bf k}}{{\bf l}}:=\prod_{i=1}^d\binom{k_i}{l_i}$.

We use the notation $A\lesssim B$ for two functions $A(x)$ and $B(x)$ of a variable $x$, if there exists a constant $c>0$ independent of $x$ such that $A(x)\le cB(x)$ for any $x$.

%%%%%%%%%%%%%%%%%%%%%%%%%%%%%%%%%%%%
\section{Preliminaries}\label{section:preliminary}
%%%%%%%%%%%%%%%%%%%%%%%%%%%%%%%%%%%%

In this section, we introduce some function spaces and prove important inequalities used throughout this paper.
Until Section \ref{section:schauder}, we fix a nonnegative measurable function $G:\bbR^d\to\bbR$
and define for any $t>0$,
$$
G_t(x)=t^{-|\mfs|/\ell}G\big(t^{-\mfs/\ell}x\big).
$$

%%%%%%%%%%%%%%%%%%%%%%%%%%%%%%%%%%%%
\subsection{Weighted Besov space}
%%%%%%%%%%%%%%%%%%%%%%%%%%%%%%%%%%%%

In this subsection, we recall from \cite{Ho23} some basics of Besov norms associated with the operator semigroup. 
For simplicity, we consider only $L^\infty$ type norms.

\begin{defi}
A continuous function $w:\bbR^d\to[0,1]$ which is strictly positive outside a set of Lebesgue measure $0$ is called a \emph{weight}.
For any weight $w$, we define the weighted $L^\infty$ norm of a measurable function $f:\bbR^d\to\bbR$ by
\begin{align*}
\|f\|_{L^\infty(w)}:=\|fw\|_{L^\infty(\bbR^d)}.
\end{align*}
We denote by $L^\infty(w)$ the space of all measurable functions with finite $L^\infty(w)$ norms, and define $C(w)=C(\bbR^d)\cap L^\infty(w)$.
\end{defi}

While we assumed that $w(x)>0$ for every $x\in\bbR^d$ in \cite{Ho23}, we impose a weaker condition to consider a weight vanishing on the hyperplane $\{0\}\times\bbR^{d-1}$ in next subsection.
Note that $\|\cdot\|_{L^\infty(w)}$ is nondegenerate because $w(x)>0$ for almost every $x\in\bbR^d$.
If $w(x)>0$ for any $x\in\bbR^d$, then $C(w)$ is a closed subspace of $L^\infty(w)$.

%In \cite{Ho23}, weights of the following type are considered.

\begin{defi}\label{defweight}
A weight $w$ is said to be \emph{$G$-controlled} if $w(x)>0$ for any $x\in\bbR^d$ and there exists a continuous function $\wco:\bbR^d\to[1,\infty)$ such that
\begin{align}\label{weightmoderate}
w(x+y)\le \wco(x)w(y)
\end{align}
for any $x,y\in\bbR^d$ and
\begin{align}\label{weightintegrable}
\sup_{0<t\le T}\sup_{x\in\bbR^d}\Big\{\|x\|_\mfs^n\,\wco\big(t^{\mfs/\ell}x\big)G(x)\Big\}<\infty
\end{align}
for any $n\ge0$ and $T>0$.
%We call $w^*$ a \emph{moderate function} of $w$.
\end{defi}

From the properties \eqref{weightmoderate} and \eqref{weightintegrable}, we have that
\begin{align}\label{Gtfboundedbyf}
\|G_t*f\|_{L^\infty(w)}\lesssim\|f\|_{L^\infty(w)}
\end{align}
uniformly over $f\in L^\infty(w)$ and $t\in(0,T]$ for any $T>0$.
This is a particular case of \cite[Lemma 2.4]{Ho23}.
Next we introduce a semigroup of integral operators.

\begin{defi}\label{asmp1}
We call a family of continuous functions $\{Q_t:\bbR^d\times\bbR^d\to\bbR\}_{t>0}$ a \emph{$G$-type semigroup} if it satisfies the following properties.
\begin{enumerate}
\renewcommand{\theenumi}{(\roman{enumi})}
\renewcommand{\labelenumi}{(\roman{enumi})}
\item\label{asmp1:semigroup}
(Semigroup property)
For any $0<s<t$ and $x,y\in\bbR^d$,
$$
\int_{\bbR^d}Q_{t-s}(x,z)Q_s(z,y)dz=Q_t(x,y).
$$
\item\label{asmp1:conservative}
(Conservativity)
For any $x\in\bbR^d$,
$$
\lim_{t\downarrow0}\int_{\bbR^d}Q_t(x,y)dy=1.
$$
\item\label{asmp1:gauss}
(Upper $G$-type estimate)
There exists a constant $C_1>0$ such that, for any $t>0$ and $x,y\in\bbR^d$,
$$
|Q_t(x,y)|\le C_1G_t(x-y).
$$
\item\label{asmp2:derivative}
(Time derivative)
For any $x,y\in\bbR^d$, $Q_t(x,y)$ is differentiable with respect to $t$. Moreover, there exists a constant $C_2>0$ such that, for any $t>0$ and $x,y\in\bbR^d$,
$$
|\partial_tQ_t(x,y)|\le C_2\, t^{-1}G_t(x-y).
$$
\end{enumerate}
\end{defi}

We fix a $G$-type semigroup $\{Q_t\}_{t>0}$ until Section \ref{section:schauder}.
If $w$ is a $G$-controlled weight, the linear operator on $L^\infty(w)$ defined by
\begin{align*}%\label{def:opQt}
(Q_tf)(x):=:Q_t(x,f):=\int_{\bbR^d}Q_t(x,y)f(y)dy\qquad (f\in L^\infty(w),\ x\in\bbR^d)
\end{align*}
is bounded in $L^\infty(w)$ uniformly over $t\in(0,1]$, by Definition \ref{asmp1}-\ref{asmp1:gauss} and the inequality \eqref{Gtfboundedbyf}. 
As an important fact, $Q_tf$ is a continuous function for any $f\in L^\infty(w)$ and $t>0$.
Moreover, if $f\in C(w)$, we have
\begin{align}\label{Qtfconvergestof}
\lim_{t\downarrow0}(Q_tf)(x)=f(x)
\end{align}
for any $x\in\bbR^d$. See \cite[Proposition 2.8]{Ho23} for the proofs.

\begin{defi}\label{def:BesovassociatedQ}
Let $w$ be a $G$-controlled weight and let $\{Q_t\}_{t>0}$ be a $G$-type semigroup.
For every $\alpha\le0$, we define the Besov space $C^{\alpha,Q}(w)$ as the completion of $C(w)$ under the norm
$$
\|f\|_{C^{\alpha,Q}(w)}
:=\sup_{0<t\le1}t^{-\alpha/\ell}\|Q_t f\|_{L^\infty(w)}.
$$
\end{defi}

By the property \eqref{Qtfconvergestof}, the norm $\|\cdot\|_{C^{\alpha,Q}(w)}$ is nondegenerate on $C(w)$.
When $\mfs=(1,1,\dots,1)$, $\ell=2$, and $Q_t$ is the heat semigroup $e^{t\Delta}$, the above norm (with $\alpha<0$ and $w=1$) is equivalent to the classical Besov norm in Euclidean setting, see e.g., \cite[Theorem 2.34]{BCD}. For more general semigroups, a similar equivalence is obtained by \cite[Theorem 5.1]{BDY} when the adjoint operator of $Q_t$ also satisfies some conditions similar to those in Definition \ref{asmp1}. As far as the authors know, without such an additional assumption for the semigroup, it is unclear whether the equivalence holds even for the case of isotropic scaling and no weight.

\begin{rem}\label{remark:Qtextension}
As stated in \cite[Proposition 2.14]{Ho23}, for any $\alpha_1<\alpha_2\le0$, the identity $\iota_{\alpha_1}:C(w)\hookrightarrow C^{\alpha_1,Q}(w)$ is uniquely extended to the continuous injection 
$$
\iota_{\alpha_1}^{\alpha_2}:C^{\alpha_2,Q}(w)\hookrightarrow C^{\alpha_1,Q}(w).
$$
Moreover, for any $\alpha\le0$, the operator $Q_t:C(w)\to C(w)$ is continuously extended to the operator $Q_t^\alpha:C^{\alpha,Q}(w)\to C(w)$ and they satisfy the relation
$$
Q_t^{\alpha_1}\circ\iota_{\alpha_1}^{\alpha_2}=Q_t^{\alpha_2}
$$
for any $\alpha_1<\alpha_2\le0$. For this compatibility, we can omit the letter $\alpha$ and use the notation $Q_t$ to mean its extension $Q_t^\alpha$ regardless of its domain.
\end{rem}

%%%%%%%%%%%%%%%%%%%%%%%%%%%%%%%%%%%%
\subsection{Temporal weights}\label{2:sec:timeweight}
%%%%%%%%%%%%%%%%%%%%%%%%%%%%%%%%%%%%

In what follows, the first variable $x_1$ in $x=(x_1, x_2, \dots, x_d)\in\bbR^d$ is regarded as the temporal variable, and the others $(x_2, \dots, x_d)$ are spatial variables, denoted by $x' = (x_2, \dots, x_d)$.
Accordingly, we denote $\mfs'=(\mfs_2,\dots,\mfs_d)$.
The aim of this paper is to extend the results in \cite{Ho23} to norms allowing a singularity at the hyperplane $\{0\}\times\bbR^{d-1}$.
We define the weight $\varphi :\bbR^d\to[0,1]$ by
$$
\varphi (x):=|x_1|^{1/\mfs_1}\wedge1
$$
and set $\varphi (x,y):=\varphi (x) \wedge \varphi (y)$.
The following inequalities are used frequently throughout this paper.

\begin{lem}\label{lemma:molifysingularity}
Let $w$ be a $G$-controlled weight.
For any $\alpha\ge0$ and $\beta\in[0,\mfs_1)$, there exists a constant $C$ such that, for any $t\in(0,1]$ and $x\in\bbR^d$ we have
\begin{align*}
\int_{\bbR^d}\varphi (y)^{-\beta} \|x-y\|_{\mfs}^\alpha\, w^\ast (x-y) G_t(x-y)dy
\le C t^{\alpha/\ell} \big\{\varphi (x)^{-\beta}\wedge t^{-\beta/\ell}\big\}
\end{align*}
and
\begin{align*}
\int_{\bbR^d}\varphi (x,y)^{-\beta} \|x-y\|_{\mfs}^\alpha\, w^\ast (x-y) G_t(x-y)dy
\le C t^{\alpha/\ell} \varphi (x)^{-\beta}.
\end{align*}
\end{lem}

\begin{proof}
The second inequality immediately follows from the first one because of the trivial inequality $\varphi (x,y)^{-\beta}\le\varphi (x)^{-\beta}+\varphi (y)^{-\beta}$.
Hence we focus on the first inequality. To obtain the bound $Ct^{(\alpha-\beta)/\ell}$, we divide the integral into two parts.
In the region $\{|y_1|^{1/\mfs_1}> t^{1/\ell}\}$, since $\varphi (y)^{-\beta}\le t^{-\beta/\ell}$ we have
\begin{align*}
    & \int_{|y_1|^{1/\mfs_1}> t^{1/\ell}}\varphi (y)^{-\beta}\|x-y\|_{\mfs}^\alpha\, w^\ast (x-y) G_t(x-y)dy\\
    & \le t^{-\beta/\ell} \int_{\bbR^d} \|z\|_{\mfs}^\alpha\, w^\ast (z) G_t(z)dz\\
    & \le t^{(\alpha-\beta)/\ell} \int_{\bbR^d} \|z\|_{\mfs}^\alpha\, w^\ast \big(t^{\mfs/\ell}z\big) G(z)dz
    \lesssim t^{(\alpha-\beta)/\ell}.
\end{align*}
In the region $\{|y_1|^{1/\mfs_1}\le t^{1/\ell}\}$, by treating the temporal variable and spatial variables separately, we have
\begin{align*}
    & \int_{|y_1|^{1/\mfs_1} \le t^{1/\ell}} \varphi (y)^{-\beta} \|x-y\|_{\mfs}^\alpha\, w^\ast (x-y) G_t(x-y)dy \\
    & \le \bigg(\int_{|y_1|^{1/\mfs_1}\le t^{1/\ell}} |y_1|^{-\beta/\mfs_1} dy_1 \bigg) 
    \bigg(\int_{\bbR^{d-1}} \sup_{z_1 \in \bbR} \|(z_1,z')\|^\alpha_{\mfs}\, w^\ast (z_1,z') G_t(z_1,z') dz' \bigg)\\
    & \lesssim (t^{\mfs_1/\ell})^{1-\beta/\mfs_1}\\
     & \hspace{35pt} \times  \bigg(t^{-\mfs_1/\ell}\int_{\bbR^{d-1}} \sup_{z_1 \in \bbR} \|(t^{\mfs_1/\ell}z_1,t^{\mfs'/\ell}z')\|^\alpha_{\mfs}\, w^\ast (t^{\mfs_1/\ell}z_1,t^{\mfs'/\ell}z') G(z_1,z') dz' \bigg)\\
    & = (t^{\mfs_1/\ell})^{1-\beta/\mfs_1}
    \bigg(t^{-\mfs_1/\ell+\alpha/\ell}\int_{\bbR^{d-1}} \sup_{z_1 \in \bbR} \|(z_1,z')\|^\alpha_{\mfs}\, w^\ast (t^{\mfs_1/\ell}z_1,t^{\mfs'/\ell}z') G(z_1,z') dz' \bigg)\\
    & \lesssim t^{(\alpha-\beta)/\ell}.
\end{align*}
Therefore, we obtain the upper bound $Ct^{(\alpha-\beta)/\ell}$.
Moreover, by decomposing
\begin{align}\label{ineq:omegashift}
\varphi (x)^{\beta}\lesssim|x_1-y_1|^{\beta/\mfs_1}+\varphi (y)^{\beta}
\lesssim\|x-y\|_\mfs^\beta+\varphi (y)^{\beta},
\end{align}
we have
\begin{align*}
&\varphi (x)^{\beta}\int_{\bbR^d}\varphi (y)^{-\beta} \|x-y\|_{\mfs}^\alpha\, w^\ast (x-y) G_t(x-y) dy\\
&\lesssim\int_{\bbR^d}\{\varphi (y)^{-\beta} \|x-y\|_{\mfs}^{\alpha+\beta}+\|x-y\|_\mfs^\alpha\}\, w^\ast (x-y) G_t(x-y) dy\\
&\lesssim t^{\alpha/\ell}.
\end{align*}
This yields another bound $Ct^{\alpha/\ell}\varphi (x)^{-\beta}$.
\end{proof}

From the above lemma, we obtain an inequality similar to \eqref{Gtfboundedbyf}.

\begin{cor}\label{proposition:Gtboundedtildew}
Let $w$ be a $G$-controlled weight.
For any $\beta\in[0,\mfs_1)$, there exists a constant $C$ such that, for any $f\in L^\infty(\varphi ^\beta w)$ we have
%$$
%w(x)|(G_t*f)(x)|\le C\big\{\varphi (x)^{-\beta}\wedge t^{-\beta/\ell}\big\}\|f\|_{L^\infty(\varphi ^\beta w)}.
%$$
%In particular, we have
$$
\sup_{0<t\le1}\|G_t*f\|_{L^\infty(\varphi ^\beta w)}
+\sup_{0<t\le1}t^{\beta/\ell}\|G_t*f\|_{L^\infty(w)}
\le C \|f\|_{L^\infty(\varphi ^\beta w)}.
$$
\end{cor}

\begin{proof}
By Lemma \ref{lemma:molifysingularity}, we have
\begin{align*}
w(x)|(G_t*f)(x)|
&\le\int_{\bbR^d}\varphi (y)^{-\beta}w^*(x-y)G_t(x-y)
\varphi (y)^\beta w(y)|f(y)|dy\\
&\le C \big\{\varphi (x)^{-\beta}\wedge t^{-\beta/\ell}\big\}\|f\|_{L^\infty(\varphi ^\beta w)}.
\end{align*}
\end{proof}

We obtain the following assertions by arguments similar to \cite{Ho23}.

\begin{prop}\label{proposition:tildew:importantfacts}
Let $w$ be a $G$-controlled weight and let $\{Q_t\}_{t>0}$ be a $G$-type semigroup.
We consider the weight $\tilde{w}:=\varphi ^\beta w$ for any fixed $\beta\in[0,\mfs_1)$.
\begin{enumerate}
\renewcommand{\theenumi}{(\roman{enumi})}
\renewcommand{\labelenumi}{(\roman{enumi})}
\item\label{proposition:tildew:importantfacts:Qtregularize}
For any $f\in L^\infty(\tilde{w})$ and $t>0$, the function $Q_tf$ belongs to $C(w)$.
\item\label{proposition:tildew:importantfacts:defofnorm}
For any $\alpha\le0$, the Besov norm
$$
\|f\|_{C^{\alpha,Q}(\tilde{w})}
:=\sup_{0<t\le1}t^{-\alpha/\ell}\|Q_t f\|_{L^\infty(\tilde{w})}
$$
is nondegenerate on $C(\tilde{w})$, so we can define $C^{\alpha,Q}(\tilde{w})$ as the completion of $C(\tilde{w})$ under this norm.
%\item
%The operator $Q_t$ is continuously extended to the operator $Q_t^\alpha:C^{\alpha,Q}(\tilde{w})\to C(\tilde{w})$ and it holds that
%\begin{align}\label{2:eq:besovnormex}
%\|f\|_{C^{\alpha,Q}(\tilde{w})}=\sup_{0<t\le1}t^{-\alpha/\ell}\|Q_t^\alpha f\|_{L^\infty(\tilde{w})}
%\end{align}
%for any $f\in C^{\alpha,Q}(\tilde{w})$.
\item\label{proposition:tildew:importantfacts:embedding}
For any $\alpha_1<\alpha_2\le0$, the identity $\tilde{\iota}_{\alpha_1}:C(\tilde{w})\hookrightarrow C^{\alpha_1,Q}(\tilde{w})$ is uniquely extended to the continuous injection $\tilde{\iota}_{\alpha_1}^{\alpha_2}:C^{\alpha_2,Q}(\tilde{w})\hookrightarrow C^{\alpha_1,Q}(\tilde{w})$.
For any $\alpha\le0$, the operator $Q_t:C(\tilde{w})\to C(\tilde{w})$ is continuously extended to the operator $\tilde{Q}_t^\alpha:C^{\alpha,Q}(\tilde{w})\to \overline{C(\tilde{w})}$, where $\overline{C(\tilde{w})}$ is the closure of $C(\tilde{w})$ under the norm $\|\cdot\|_{L^\infty(\tilde{w})}$.
Moreover, they satisfy $\tilde{Q}_t^{\alpha_1}\circ\tilde{\iota}_{\alpha_1}^{\alpha_2}=\tilde{Q}_t^{\alpha_2}$ for any $\alpha_1<\alpha_2\le0$.
\item\label{proposition:tildew:importantfacts:embeddingwtildew}
For any $\alpha\le0$, the identity $i:C(w)\hookrightarrow C(\tilde{w})$ is uniquely extended to the continuous injection $i_\alpha:C^{\alpha,Q}(w)\hookrightarrow C^{\alpha,Q}(\tilde{w})$.
Moreover, the extensions $\tilde{Q}_t^\alpha:C^{\alpha,Q}(\tilde{w})\to \overline{C(\tilde{w})}$ and $Q_t^\alpha:C^{\alpha,Q}(w)\to C(w)$ defined in \ref{proposition:tildew:importantfacts:embedding} and Remark \ref{remark:Qtextension} satisfy the relation
$$
i\circ Q_t^\alpha=\tilde{Q}_t^\alpha\circ i_\alpha.
$$
Consequently, we can use the same notation $Q_t$ to denote both $Q_t^\alpha$ and $\tilde{Q}_t^\alpha$.
\item\label{proposition:tildew:importantfacts:contit}
For any $\alpha\le0$, there exists a constant $C>0$ such that, for any $f\in C^{\alpha,Q}(\tilde{w})$, $t\in(0,1]$, and $\varepsilon\in[0,\ell]$, we have
$$
\|(Q_t-\id)f\|_{C^{\alpha-\varepsilon,Q}(\tilde{w})}
\le C\, t^{\varepsilon/\ell}\|f\|_{C^{\alpha,Q}(\tilde{w})}.
$$
\end{enumerate}
\end{prop}

The norm $C^{\alpha,Q}(\varphi ^\beta w)$ is used in the proof of Theorem \ref{theorem:singularreconstruction}.

\begin{proof}
\ref{proposition:tildew:importantfacts:Qtregularize}
We have $Q_tf\in L^\infty(w)$ by Corollary \ref{proposition:Gtboundedtildew}.
To show the continuity of $(Q_tf)(x)$ with respect to $x$, it is sufficient to consider the case $t=1$. 
By the property \eqref{weightintegrable}, for any fixed $R>0$ and $n\ge0$, the inequalities
\begin{align*}
w(x)|Q_1(x,y)f(y)|
&\lesssim w^*(x-y)w(y)G(x-y)|f(y)|
\lesssim \frac{\varphi (y)^{-\beta}}{1+\|y\|_\mfs^n}\|f\|_{L^\infty(\tilde{w})}
\end{align*}
hold uniformly over $\|x\|_\mfs\le R$ and $y\in\bbR^d$.
Since $\int_{\bbR^d}\varphi (y)^{-\beta}/(1+\|y\|_\mfs^n)dy<\infty$ for $n>|\mfs|$, we have
\begin{align*}
\lim_{z\to x}(Q_1f)(z)w(z)=\int_{\bbR^d}\lim_{z\to x}Q_1(z,y)f(y)w(z)dy=(Q_1f)(x)w(x)
\end{align*}
by Lebesgue's convergence theorem. 
Since $w$ is strictly positive and continuous, we have $\lim_{z\to x}(Q_1f)(z)=(Q_1f)(x)$.

\medskip

\ref{proposition:tildew:importantfacts:defofnorm}
It is sufficient to show that
$$
\lim_{t\downarrow0}(Q_tf)(x)=f(x)
$$
for any $f\in C(\tilde{w})$ and $x\in\bbR^d$.
For any $\varepsilon>0$, we can choose $\delta>0$ such that $|f(y)-f(x)|<\varepsilon$ if $\|y-x\|_\mfs<\delta$, and have

\begin{align*}
&|w(x)(Q_tf-f)(x)|\\
&=w(x)\bigg|\int_{\bbR^d}Q_t(x,y)\big(f(y)-f(x)\big)dy+\bigg(\int_{\bbR^d}Q_t(x,y)dy-1\bigg)f(x)\bigg|\\
&\le w(x)\varepsilon\int_{\|y-x\|_\mfs<\delta}G_t(x-y)dy+w(x)\int_{\|y-x\|_\mfs\ge\delta}G_t(x-y)|f(y)|dy\\
&\quad+w(x)|f(x)|\int_{\|y-x\|_\mfs\ge\delta}G_t(x-y)dy+w(x)|f(x)|\bigg|\int_{\bbR^d}Q_t(x,y)dy-1\bigg|.
\end{align*}

In the far right-hand side, the only nontrivial part is the second term. We bound it from above by
\begin{align*}
&\int_{\|y-x\|_\mfs\ge\delta}G_t(x-y)\wco(x-y)|f(y)|w(y)dy\\
&\le\|f\|_{L^\infty(\tilde{w})}\int_{\|y-x\|_\mfs\ge\delta}\varphi (y)^{-\beta}w^*(x-y)G_t(x-y)dy\\
&\le\|f\|_{L^\infty(\tilde{w})}\delta^{-\mfs_1}\int_{\bbR^d}\|y-x\|_\mfs^{\mfs_1}\varphi (y)^{-\beta}w^*(x-y)G_t(x-y)dy\\
&\lesssim\|f\|_{L^\infty(\tilde{w})}\delta^{-\mfs_1}t^{(\mfs_1-\beta)/\mfs_1}.
\end{align*}
Since $\beta<\mfs_1$, we obtain the convergence as $t\downarrow0$.

\medskip

The proofs of \ref{proposition:tildew:importantfacts:embedding} and \ref{proposition:tildew:importantfacts:embeddingwtildew} are similar to \cite[Proposition 2.14]{Ho23},
and the proof of \ref{proposition:tildew:importantfacts:contit} is similar to \cite[Lemma 2.15]{Ho23}.
\end{proof}

%%%%%%%%%%%%%%%%%%%%%%%%%%%%%%%%%%%%
\section{Reconstruction of singular modelled distributions}\label{section:singular}
%%%%%%%%%%%%%%%%%%%%%%%%%%%%%%%%%%%%

In this section, we recall from \cite{Hai14} the definitions of regularity structures, models, and singular modelled distributions, and prove the reconstruction theorem for singular modelled distributions using the operator semigroup.
For simplicity, we consider only regularity structures, rather than general regularity-integrability structures as in \cite{Ho23}.
Throughout this and next sections, we fix a $G$-type semigroup $\{Q_t\}_{t>0}$.

%%%%%%%%%%%%%%%%%%%%%%%%%%%%%%%%%%%%
\subsection{Regularity structures and models}
%%%%%%%%%%%%%%%%%%%%%%%%%%%%%%%%%%%%

\begin{defi}\label{*def:ATG}
A \emph{regularity structure} $\scT=(\bfA,\bfT,\bfG)$ consists of the following objects.
\begin{itemize}
\item[(1)]
(Index set) $\bfA$ is a locally finite subset of $\bbR$ bounded below.
\item[(2)]
(Model space) $\bfT=\bigoplus_{\alpha\in \bfA}\bfT_\alpha$ is an algebraic sum of Banach spaces $(\bfT_\alpha,\|\cdot\|_\alpha)$.
\item[(3)]
(Structure group) $\bfG$ is a group of continuous linear operators on $\bfT$ such that, for any $\Gamma\in \bfG$ and $\alpha\in\bfA$,
$$
(\Gamma-\id)\bfT_\alpha\subset\bfT_{<\alpha}:=\bigoplus_{{\beta}\in \bfA,\, {\beta} < \alpha }\bfT_\beta.
$$
\end{itemize}
The smallest element $\alpha_0$ of $\bfA$ is called the \emph{regularity} of $\scT$.
For any $\alpha\in\bfA$, we denote by $P_\alpha:\bfT\to\bfT_\alpha$ the canonical projection and write
$$
\|\tau\|_{\alpha}:=\|P_{\alpha}\tau\|_{\alpha}
$$
for any $\tau\in\bfT$, by abuse of notation.
\end{defi}

Following \cite{Ho23}, we define the topology on the space of models by using $\{Q_t\}_{t>0}$.
For two Banach spaces $X$ and $Y$, we denote by $\mcL(X,Y)$ the Banach space of all continuous linear operators $X\to Y$. When $Y=\bbR$, we write $X^*:=\mcL(X,\bbR)$.

\begin{defi}\label{def:model}
Let $w$ be a $G$-controlled weight.
A \emph{smooth model} $M=(\Pi,\Gamma)$ is a pair of two families of continuous linear operators $\Pi=\{\Pi_x:\bfT\to C(w)\}_{x\in\bbR^d}$ and $\Gamma=\{\Gamma_{xy}\}_{x,y\in\bbR^d}\subset \bfG$ with the following properties.
\begin{itemize}
\item[(1)]
(Algebraic conditions)
$\Pi_x\Gamma_{xy}=\Pi_y$, $\Gamma_{xx}=\id$, and $\Gamma_{xy}\Gamma_{yz}=\Gamma_{xz}$ for any $x,y,z\in\bbR^d$.
\item[(2)]
(Analytic conditions)
For any $\gamma\in\bbR$,
\begin{align*}
\|\Pi\|_{\gamma,w}&:=\max_{\alpha\in\bfA,\, \alpha<\gamma}\,
\sup_{0<t\le1}\, \sup_{x\in\bbR^d}
\Big(t^{-\alpha/\ell}w(x)\big\|Q_t\big(x,\Pi_x(\cdot)\big)\big\|_{\bfT_\alpha^*}\Big)
\\
&=\max_{\alpha\in\bfA,\, \alpha<\gamma}\,
\sup_{0<t\le1}\, \sup_{x\in\bbR^d}\, \sup_{\tau\in \bfT_{\alpha}\setminus\{0\}}
\Bigg(t^{-\alpha/\ell}w(x)
\frac{|Q_t(x,\Pi_x\tau)|}{\|\tau\|_\alpha}
\Bigg)<\infty
\end{align*}
and
\begin{align*}
\|\Gamma\|_{\gamma,w}
&:=\max_{\substack{\alpha,\beta\in\bfA\\ \beta<\alpha<\gamma}}\,
\sup_{x,y\in\bbR^d,\,x\neq y}
\frac{w(x)\|\Gamma_{yx}\|_{\mcL(\bfT_{\alpha},\bfT_{\beta})}}{w^\ast(y-x)\|y-x\|_\mfs^{\alpha-\beta}}\\
&=\max_{\substack{\alpha,\beta\in\bfA\\ \beta<\alpha<\gamma}}\,
\sup_{x,y\in\bbR^d,\,x\neq y}\, \sup_{\tau\in \bfT_{\alpha}\setminus\{0\}}
\frac{w(x)\|\Gamma_{yx}\tau\|_\beta}{w^\ast(y-x)\|y-x\|_\mfs^{\alpha-\beta}\|\tau\|_\alpha}<\infty.
\end{align*}
\end{itemize}
We write $\tri M\tri_{\gamma,w}:=\|\Pi\|_{\gamma,w}+\|\Gamma\|_{\gamma,w}$.
In addition, for any two smooth models $M^{(i)}=(\Pi^{(i)},\Gamma^{(i)})$ with $i\in\{1,2\}$, we define the pseudo-metrics
$$
\tri M^{(1)};M^{(2)}\tri_{\gamma,w}:=\|\Pi^{(1)}-\Pi^{(2)}\|_{\gamma,w}+\|\Gamma^{(1)}-\Gamma^{(2)}\|_{\gamma,w}
$$
by replacing $\Pi$ and $\Gamma$ above with $\Pi^{(1)}-\Pi^{(2)}$ and $\Gamma^{(1)}-\Gamma^{(2)}$ respectively.
Finally, we define the space $\scM_w(\scT)$ as the completion of the set of all smooth models, under the pseudo-metrics $\tri\cdot;\cdot\tri_{\gamma,w}$ for all $\gamma\in\bbR$.
We call each element of $\scM_w(\scT)$ a \emph{model} for $\scT$. We still use the notation $M=(\Pi,\Gamma)$ to denote a generic model.
\end{defi}

When $\ell=2$ and $Q_t$ is the heat semigroup $e^{t\Delta}$, the above definition essentially coincides with the original definition of models \cite[Definition 2.17]{Hai14} if we ignore the difference between local and global bounds. For more general semigroups, such an equivalence is unclear by the same reason as the case of Besov norms.

\begin{rem}\label{proposition:whatisPixtau}
As stated in \cite[Proposition 3.3]{Ho23}, if there exist two $G$-controlled weights $w_1$ and $w_2$ that satisfy
$$
\sup_{x\in\bbR^d}\big\{\|x\|_\mfs^n\,\wco(x)w_1(x)\big\}
+\sup_{x\in\bbR^d}\big\{\|x\|_\mfs^n\,\wco_1(x)w_2(x)\big\}
<\infty
$$
for any $n\ge0$, and such that $ww_1$ and $ww_2$ are also $G$-controlled, then we can regard $\Pi_x$ as a continuous linear operator from $\bfT$ to $C^{\alpha_0\wedge0,Q}(ww_1)$, where $\alpha_0$ is the regularity of $\scT$.
More precisely, for any $\alpha<\gamma$ and $\tau\in\bfT_\alpha$ we have
$$
\sup_{x\in\bbR^{d}}(ww_2)(x)\|\Pi_x\tau\|_{C^{\alpha_0\wedge0,Q}(ww_1)}
\lesssim\|\Pi\|_{\gamma,w}(1+\|\Gamma\|_{\gamma,w})\|\tau\|_\alpha.
$$
In what follows, we assume the existence of $w_1$ and $w_2$ as above, and regard $\Pi_x\tau$ as an element of $C^{\alpha_0\wedge0,Q}(ww_1)$ for any $\tau\in\bfT$.
\end{rem}

%%%%%%%%%%%%%%%%%%%%%%%%%%%%%%%%%%%%
\subsection{Singular modelled distributions}
%%%%%%%%%%%%%%%%%%%%%%%%%%%%%%%%%%%%

Throughout the rest of this section, we fix a regularity structure $\scT$ of regularity $\alpha_0$, and also fix $G$-controlled weights $w$ and $v$ such that $wv$ is also $G$-controlled.
Recall the definitions of functions $\varphi (x)$ and $\varphi (x,y)$ from Section \ref{2:sec:timeweight}.

\begin{defi}\label{def:besovsMD}
Let $M = (\Pi,\Gamma) \in \scM_w(\scT)$. For any $\gamma\in\bbR$ and $\eta\le\gamma$, we define $\mcD_v^{\gamma,\eta}(\Gamma)$ as the space of all functions $f:(\bbR\setminus\{0\})\times\bbR^{d-1}\to \bfT_{<\gamma}$ such that
\begin{align*}
\lp f\rp_{\gamma,\eta,v}&:=\max_{\alpha<\gamma}\sup_{x\in(\bbR\setminus\{0\})\times\bbR^{d-1}}
\frac{v(x)\|f(x)\|_\alpha}{\varphi (x)^{(\eta-\alpha)\wedge0}}<\infty,\\
\| f\|_{\gamma,\eta,v}
&:=\max_{\alpha<\gamma}\sup_{\substack{x,y\in(\bbR\setminus\{0\})\times\bbR^{d-1},\,x\neq y \\ \|y-x\|_\mfs\le\varphi (x,y)}}
\frac{v(x)\|\Delta^\Gamma_{yx}f\|_\alpha}{\varphi (x,y)^{\eta-\gamma}\,v^*(x-y)\|y-x\|_\mfs^{\gamma-\alpha}}
<\infty,
\end{align*}
where $\Delta^\Gamma_{yx}f := f(y) - \Gamma_{yx}f(x)$. 
We write $\tri f\tri_{\gamma,\eta,v}:=\lp f\rp_{\gamma,\eta,v}+\|f\|_{\gamma,\eta,v}$.
We call each element of $\mcD_v^{\gamma,\eta}(\Gamma)$ a \emph{singular modelled distribution}.

In addition, for any two models $M^{(i)}=(\Pi^{(i)},\Gamma^{(i)})\in\scM_{w}(\scT)$ and singular modelled distributions $f^{(i)}\in\mcD_v^{\gamma,\eta}(\Gamma^{(i)})$ with $i\in\{1,2\}$, we define $\tri f^{(1)};f^{(2)}\tri_{\gamma,\eta,v}:=\lp f^{(1)}-f^{(2)}\rp_{\gamma,\eta,v}+\|f^{(1)};f^{(2)}\|_{\gamma,\eta,v}$ by
\begin{align*}
\lp f^{(1)}-f^{(2)}\rp_{\gamma,\eta,v}&:=\max_{\alpha<\gamma}\sup_{x\in(\bbR\setminus\{0\})\times\bbR^{d-1}}
\frac{v(x)\|f^{(1)}(x) - f^{(2)}(x)\|_\alpha}{\varphi (x)^{(\eta-\alpha)\wedge0}},\\
\| f^{(1)};f^{(2)}\|_{\gamma,\eta,v}
&:=\max_{\alpha<\gamma}\sup_{\substack{x,y\in(\bbR\setminus\{0\})\times\bbR^{d-1},\, x\neq y \\ \|y-x\|_\mfs\le\varphi (x,y)}}
\frac{v(x)\|\Delta^\Gamma_{yx}f^{(1)} - \Delta^\Gamma_{yx}f^{(2)} \|_\alpha}
{\varphi (x,y)^{\eta-\gamma}\,v^*(x-y)\|y-x\|_\mfs^{\gamma-\alpha}}.
\end{align*}
\end{defi}

In \cite{Hai14}, the topologies of the space of models and the space of modelled distributions are defined by the family of pseudo-metrics parametrized by compact subsets $K$ of $\bbR^d$, where $x$ and $y$ in the above definitions are restricted within $K$.
In this paper, we employ weight functions $w$ and $v$ instead of such local bounds.
%As mentioned in Section \ref{sec:intro}, we have to impose global estimates in order to use an integral operator $Q_t$ defined over the entire spacetime.

We consider the relations between $\mcD_v^{\gamma,\eta}$ under varying parameters $\gamma,\eta$, as well as the relation between $\mcD_v^{\gamma,\eta}$ and a variant.
We say that the function $u:\bbR^d\to\bbR$ is \emph{symmetric} if $u(-x)=u(x)$ for any $x\in\bbR^d$.

\begin{prop}\label{def:variantsnorms}
Let $M = (\Pi,\Gamma) \in \scM_w(\scT)$ and $\eta\le\gamma$.
\begin{enumerate}
\renewcommand{\theenumi}{(\roman{enumi})}
\renewcommand{\labelenumi}{(\roman{enumi})}
\item\label{def:variantsnorms-1}
For any $\theta\le\eta$, we have the continuous embedding $\mcD_v^{\gamma,\eta}(\Gamma)\hookrightarrow\mcD_v^{\gamma,\theta}(\Gamma)$.
\item\label{def:variantsnorms0}
Assume that $w^*$ is symmetric.
For each $\alpha\in\bbR$, we denote by $P_{<\alpha}:\bfT\to\bfT_{<\alpha}$ the canonical projection.
For any $\eta\le\delta\le\gamma$, the map $P_{<\delta}$ extends to a continuous linear map $\mcD_v^{\gamma,\eta}(\Gamma)\to\mcD_{wv}^{\delta,\eta}(\Gamma)$.
To be precise, we have the inequality
$$
\| P_{<\delta}f\|_{\delta,\eta,wv}
\lesssim\|\Gamma\|_{\gamma,w}\lp f\rp_{\gamma,\eta,v}+\| f\|_{\gamma,\eta,v}.
$$
\item\label{def:variantsnorms1}
Instead of the norm $\| f\|_{\gamma,\eta,v}$, we define
$$
\| f\|_{\gamma,\eta,v}^{\#}
:=\max_{\alpha<\gamma}\sup_{x,y\in(\bbR\setminus\{0\})\times\bbR^{d-1},\,x\neq y}
\frac{v(x)\|\Delta^\Gamma_{yx}f\|_\alpha}{\varphi (x,y)^{\eta-\gamma}\,v^*(x-y)\|y-x\|_\mfs^{\gamma-\alpha}}.
$$
Then the inequality $\| f\|_{\gamma,\eta,v}\le\| f\|_{\gamma,\eta,v}^\#$ obviously holds.
Conversely, if $w^*$ is symmetric, then we also have
$$
\|f\|_{\gamma,\eta\wedge\alpha_0,wv}^\#
\lesssim(1+\|\Gamma\|_{\gamma,w})\lp f\rp_{\gamma,\eta,v}+\| f\|_{\gamma,\eta,v}.
$$
\end{enumerate}
\end{prop}

\begin{proof}
\ref{def:variantsnorms-1}
The assertion immediately follows from the inequalities $\varphi (x)^{(\eta-\alpha)\wedge0}\le\varphi (x)^{(\theta-\alpha)\wedge0}$ and $\varphi (x,y)^{\eta-\gamma}\le\varphi (x,y)^{\theta-\gamma}$.
\medskip

\ref{def:variantsnorms0}
For any $x,y\in(\bbR\setminus\{0\})\times\bbR^{d-1}$ such that $\|y-x\|_\mfs\le\varphi (x,y)$ and any $\alpha<\delta$, we decompose
\begin{align*}
(wv)(x)\|\Delta_{yx}^\Gamma P_{<\delta}f\|_\alpha
&\le
v(x)\|\Delta_{yx}^\Gamma f\|_\alpha+(wv)(x)\sum_{\beta\in[\delta,\gamma)}\|\Gamma_{yx}P_\beta f(x)\|_\alpha\\
&=:A_1+A_2.
\end{align*}
For $A_1$, by definition of the norm $\|f\|_{\gamma,\eta,v}$ we have
\begin{align*}
A_1
&\le\|f\|_{\gamma,\eta,v}\,v^*(x-y)\varphi (x,y)^{\eta-\gamma}\|y-x\|_\mfs^{\gamma-\alpha}\\
&\le
\|f\|_{\gamma,\eta,v}\,v^*(x-y)\varphi (x,y)^{\eta-\delta}\|y-x\|_\mfs^{\delta-\alpha}.
\end{align*}
For $A_2$, by definitions of the model and the norm $\lp f\rp_{\gamma,\eta,v}$ we have
\begin{align*}
A_2&\le\sum_{\beta\in[\delta,\gamma)}w(x)\|\Gamma_{yx}\|_{\mcL(\bfT_\beta,\bfT_\alpha)}v(x)\|f(x)\|_\beta\\
&\le\|\Gamma\|_{\gamma,w}\lp f\rp_{\gamma,\eta,v}\,
w^*(y-x)\sum_{\beta\in[\delta,\gamma)}\|y-x\|_\mfs^{\beta-\alpha}\varphi (x)^{\eta-\beta}\\
&\le\|\Gamma\|_{\gamma,w}\lp f\rp_{\gamma,\eta,v}\,
w^*(x-y)\|y-x\|_\mfs^{\delta-\alpha}\varphi (x,y)^{\eta-\delta}.
\end{align*}
Thus we obtain the desired inequality for $\| P_{<\delta}f\|_{\delta,\eta,wv}$.

\medskip

\ref{def:variantsnorms1}
It is sufficient to show the estimate of $\Delta_{yx}^\Gamma f$ on the region $\|y-x\|_\mfs>\varphi (x,y)$.
For any $\alpha<\gamma$ we decompose
$$
(wv)(x)\|\Delta^\Gamma_{yx}f\|_\alpha\\
\le v(x)\|f(y)\|_\alpha+(wv)(x)\sum_{\beta\in[\alpha,\gamma)}\|\Gamma_{yx}P_\beta f(x)\|_\alpha
=:B_1+B_2.
$$
For $B_1$, by definition of the norm $\lp f\rp_{\gamma,\eta,v}$ we have
\begin{align*}
B_1
&\le v^*(x-y)v(y)\|f(y)\|_\alpha\\
&\le\lp f\rp_{\gamma,\eta,v}\,v^*(x-y)\varphi (y)^{(\eta-\alpha)\wedge0}\\
&\le\lp f\rp_{\gamma,\eta,v}\,v^*(x-y)\varphi (x,y)^{(\eta-\alpha)\wedge0}\\
&\le
\lp f\rp_{\gamma,\eta,v}\,v^*(x-y)\varphi (x,y)^{\eta\wedge\alpha-\gamma}\|y-x\|_\mfs^{\gamma-\alpha}.
\end{align*}
For $B_2$, by an argument similar to $A_2$ in the proof of \ref{def:variantsnorms0}, we have
\begin{align*}
B_2
&\le\|\Gamma\|_{\gamma,w}\lp f\rp_{\gamma,\eta,v}\,
w^*(y-x)\sum_{\beta\in[\alpha,\gamma)}\|y-x\|_\mfs^{\beta-\alpha}\varphi (x)^{(\eta-\beta)\wedge0}\\
&\lesssim
\|\Gamma\|_{\gamma,w}\lp f\rp_{\gamma,\eta,v}\,
w^*(x-y)\|y-x\|_\mfs^{\gamma-\alpha}
\varphi (x,y)^{\eta\wedge\alpha-\gamma}.
\end{align*}
Thus we obtain the desired inequality.
\end{proof}

%\begin{remark}\label{3.2}
%In \cite[Section 6]{Hai14}, Hairer originally introduced a weaker norm of the type $\tri f\tri_{\gamma,\eta,v}^{\Gamma,\#}:=\lp f\rp_{\gamma,\eta,v}+\|f\|_{\gamma,\eta,v}^{\Gamma,\#}$ with
%$$
%\| f\|_{\gamma,\eta,v}^{\Gamma,\#}
%:=\max_{\alpha<\gamma}\sup_{\substack{x,y\in\bbR^d\setminus\{x_1=0\},\,x\neq y \\ \|y-x\|_\mfs\le\varphi (x,y)}}
%\frac{v(x)\|\Delta^\Gamma_{y,x}f\|_\alpha}{\varphi (x,y)^{\eta-\gamma}v^\ast(x-y)\|y-x\|_\mfs^{\gamma-\alpha}}.
%$$
%(Here we ignore the difference that the bounds for $(x,y)$ are local in \cite{Hai14} but global in the above definition.)
%The inequality $\tri f\tri_{\gamma,\eta,v}^{\Gamma,\#}\le\tri f\tri_{\gamma,\eta,v}^{\Gamma}$ is trivial.
%If $w$ is symmetric in the sense that $w(-x)=w(x)$ for any $x\in\bbR^d$, then we have the converse inequality
%$$
%\tri f\tri_{\gamma,\eta\wedge\alpha_0,wv}^\Gamma
%\lesssim(1+\|\Gamma\|_{\gamma,w})\lp f\rp_{\gamma,\eta,v}+\| f\|_{\gamma,\eta,v}^{\Gamma,\#}.
%$$
%\end{remark}

We also recall the definition of reconstruction.

\begin{defi}\label{def:besovreconst}
Let $M = (\Pi,\Gamma) \in \scM_w(\scT)$. For any $\eta\le\gamma$ and $f\in\mcD_{v}^{\gamma,\eta}(\Gamma)$, we say that $\Lambda\in C^{\zeta,Q}(wv)$ with some $\zeta\le0$ is a \emph{reconstruction} of $f$ for $M$, if it satisfies
\begin{align*}
\lb\Lambda\rb_{\gamma,\eta,wv}:=\sup_{0<t\le1}\sup_{x\in(\bbR\setminus\{0\})\times\bbR^{d-1}}
\Big(t^{-\gamma/\ell}\varphi (x)^{\gamma-\eta}(wv)(x)|Q_t(x,\Lambda_x)|\Big)<\infty,
\end{align*}
where $\Lambda_x:=\Lambda-\Pi_xf(x)$.
Furthermore, for any $M^{(i)}=(\Pi^{(i)},\Gamma^{(i)})\in\scM_{w}(\scT)$, $f^{(i)}\in\mcD_{v}^{\gamma,\eta}(\Gamma^{(i)})$, and any reconstructions $\Lambda^{(i)}\in C^{\zeta,Q}(wv)$ of $f^{(i)}$ for $M^{(i)}$ with $i\in\{1,2\}$, we define
\begin{align*}
& \lb\Lambda^{(1)};\Lambda^{(2)}\rb_{\gamma,\eta,wv}\\
&:=\sup_{0<t\le1}\sup_{x\in(\bbR\setminus\{0\})\times\bbR^{d-1}}
\Big(t^{-\gamma/\ell}\varphi (x)^{\gamma-\eta}(wv)(x)\big|Q_t\big(x,\Lambda^{(1)}_x-\Lambda^{(2)}_x\big)\big|\Big),
\end{align*}
where $\Lambda_x^{(i)}:=\Lambda^{(i)}-\Pi_x^{(i)}f^{(i)}(x)$ for each $i\in\{1,2\}$.
\end{defi}

%%%%%%%%%%%%%%%%%%%%%%%%%%%%%%%%%%%%
\subsection{Reconstruction Theorem}
%%%%%%%%%%%%%%%%%%%%%%%%%%%%%%%%%%%%

In this subsection, we provide a short proof of the reconstruction theorem.
First, we prove the theorem for the subclass $\mcD_v^{\gamma,\eta}(\Gamma)^\#$ of $\mcD_v^{\gamma,\eta}(\Gamma)$ consisting of all functions $f:(\bbR\setminus\{0\})\times\bbR^{d-1}\to\bfT_{<\gamma}$ such that
$$
\tri f\tri_{\gamma,\eta,v}^\#:=\lp f\rp_{\gamma,\eta,v}+\|f\|_{\gamma,\eta,v}^\#<\infty.
$$
In addition, for any $M^{(i)}=(\Pi^{(i)},\Gamma^{(i)})\in\scM_{w}(\scT)$ and $f^{(i)}\in\mcD_v^{\gamma,\eta}(\Gamma^{(i)})^\#$ with $i\in\{1,2\}$, we define $\tri f^{(1)};f^{(2)}\tri_{\gamma,\eta,v}^\#:=\lp f^{(1)}-f^{(2)}\rp_{\gamma,\eta,v}+\|f^{(1)};f^{(2)}\|_{\gamma,\eta,v}^\#$ similarly to definition \ref{def:besovsMD}.

\begin{thm}\label{theorem:singularreconstruction}
Let $\gamma>0$ and $\eta\in(\gamma-\mfs_1,\gamma]$. Then for any $M=(\Pi,\Gamma)\in\scM_{w}(\scT)$ and $f\in\mcD_{v}^{\gamma,\eta}(\Gamma)^\#$, there exists a unique reconstruction $\mcR f\in C^{\zeta,Q}(wv)$ of $f$ for $M$ with $\zeta:=\eta\wedge\alpha_0\wedge0$ and it holds that
\begin{align}
\label{besovreconstruction}
\|\mcR f\|_{C^{\zeta,Q}(wv)}&\lesssim\|\Pi\|_{\gamma,w}\tri f\tri_{\gamma,\eta,v}^\#,\\
\label{besovreconstructioncharacterize}
\lb\mcR f\rb_{\gamma,\eta,wv}&\lesssim \|\Pi\|_{\gamma,w}\|f\|_{\gamma,\eta,v}^\#.
\end{align}
Moreover, there is an affine function $C_R>0$ of $R>0$ such that
\begin{align*}
\|\mcR f^{(1)}-\mcR f^{(2)}\|_{C^{\zeta,Q}(wv)}&\le C_R\big(\|\Pi^{(1)}-\Pi^{(2)}\|_{\gamma,w}+\tri f^{(1)};f^{(2)}\tri_{\gamma,\eta,v}^\#\big),\\
\lb\mcR f^{(1)};\mcR f^{(2)}\rb_{\gamma,\eta,wv}&\le C_R\big(\|\Pi^{(1)}-\Pi^{(2)}\|_{\gamma,w}+\| f^{(1)};f^{(2)}\|_{\gamma,\eta,v}^\#\big)
\end{align*}
for any $M^{(i)}=(\Pi^{(i)},\Gamma^{(i)})\in\scM_{w}(\scT)$ and $f^{(i)}\in\mcD_{v}^{\gamma,\eta}(\Gamma^{(i)})$ with $i\in\{1,2\}$ such that $\tri M^{(i)}\tri_{\gamma,w}\le R$ and $\tri f^{(i)}\tri_{\gamma,\eta,v}^\#\le R$.
\end{thm}

\begin{proof} 
The proof is carried out by a method similar to that of \cite[Theorem 4.1]{Ho23}, but we have to treat the temporal weight more carefully.
For $t>0$ and $0<s\le t\wedge1$, we define the functions
\begin{align*}
\mcR_s^t f(x):=
\left\{
\begin{aligned}
&\int_{\bbR^d}Q_{t-s}(x,y)Q_s\big(y,\Pi_yf(y)\big)dy,\quad&&s<t,\\
&Q_t\big(x,\Pi_xf(x)\big),&&s=t.
\end{aligned}
\right.
\end{align*}
Note that
\begin{align*}
(wv)(x)\big|Q_t\big(x,\Pi_xf(x)\big)\big|
&\le\sum_{\alpha<\gamma}w(x)\big\|Q_t\big(x,\Pi_x(\cdot)\big)\big\|_{\bfT_\alpha^*}
v(x)\|f(x)\|_\alpha\\
&\le\|\Pi\|_{\gamma,w} \lp f\rp_{\gamma,\eta,v}\sum_{\alpha<\gamma}
t^{\alpha/\ell}\varphi (x)^{(\eta-\alpha)\wedge0}.
\end{align*}
Thus, by Proposition \ref{proposition:tildew:importantfacts}-\ref{proposition:tildew:importantfacts:Qtregularize}, for any $s\in(0,t)$ we have $\mcR_s^tf\in C(wv)$ and
\begin{equation}\label{eq:proof:reconst:initial}
\|\mcR_s^tf\|_{L^\infty(wv)}
\lesssim\|\Pi\|_{\gamma,w} \lp f\rp_{\gamma,\eta,v}\sum_{\alpha<\gamma}
s^{\alpha/\ell}(t-s)^{(\eta-\alpha)\wedge0}.
\end{equation}
We separate the proof into four steps.

\medskip

\noindent
{\bf (1) Cauchy property.}\ Set $F_x:=\Pi_xf(x)$. By the definition of norms, we have
\begin{align}\label{eq:proof:reconst:coherence}
\begin{aligned}
&(wv)(y)|Q_t(x,F_y-F_x)|\\
&=(wv)(y)\big|Q_t\big(x,\Pi_x\big\{\Gamma_{xy}f(y)-f(x)\big\}\big)\big|\\
&\le w^*(y-x)\sum_{\alpha<\gamma}w(x)\|Q_t(x,\Pi_x(\cdot))\|_{\bfT_\alpha^*}
v(y)\|\Gamma_{xy}f(y)-f(x)\|_\alpha\\
&\le\|\Pi\|_{\gamma,w} \|f\|_{\gamma,\eta,v}^\# (w^*v^*)(y-x) 
\sum_{\alpha<\gamma} t^{\alpha/\ell}\varphi (x,y)^{\eta-\gamma}\|y-x\|_\mfs^{\gamma-\alpha}.
\end{aligned}
\end{align}
By the semigroup property, for any $0<u<s<t\wedge1$ we have

\begin{align*}
&(wv)(x)|\mcR_s^tf(x)-\mcR_u^tf(x)|\\
&\le \int_{(\bbR^d)^2}(w^*v^*)(x-y)(wv)(y)|Q_{t-s}(x,y)Q_{s-u}(y,z)Q_u(z,F_y-F_z)| dydz\\
&\lesssim  \|\Pi\|_{\gamma,w} \|f\|_{\gamma,\eta,v}^\#\sum_{\alpha<\gamma}u^{\alpha/\ell} \int_{(\bbR^{d})^2} (w^*v^*) (x-y) (w^*v^*)(y-z) \\
&\hspace{100pt}\times G_{t-s}(x-y)G_{s-u}(y-z)
    \varphi (y,z)^{\eta-\gamma}\|y-z\|_{\mfs}^{\gamma-\alpha} dydz.
\end{align*}

By applying the second inequality of Lemma \ref{lemma:molifysingularity} to the integral with respect to $z$
and then applying the first inequality of Lemma \ref{lemma:molifysingularity} to the integral with respect to $y$, we obtain
\begin{align*}
&(wv)(x)|\mcR_s^tf(x)-\mcR_u^tf(x)|\\
&\lesssim  \|\Pi\|_{\gamma,w} \|f\|_{\gamma,\eta,v}^\# \\
& \hspace{30pt} \times  \sum_{\alpha<\gamma}u^{\alpha/\ell} 
(s-u)^{(\gamma-\alpha)/\ell}
\int_{(\bbR^{d})^2} (w^*v^*) (x-y) G_{t-s}(x-y)\varphi (y)^{\eta-\gamma}dy\\
&\lesssim  \|\Pi\|_{\gamma,w} \|f\|_{\gamma,\eta,v}^\#\sum_{\alpha<\gamma}u^{\alpha/\ell} 
(s-u)^{(\gamma-\alpha)/\ell}\varphi (x)^{\eta-\gamma}.
\end{align*}
Consequently, when $u\in[s/2,s)$ we have the inequality
\begin{align}\label{eq:singularreconstructioncauchy}
(wv)(x)|\mcR_s^tf(x)-\mcR_u^tf(x)|
\lesssim\|\Pi\|_{\gamma,w} \|f\|_{\gamma,\eta,v}^\#\,\varphi (x)^{\eta-\gamma}s^{\gamma/\ell}.
\end{align}
Similarly to the proof of \cite[Theorem 4.1]{Ho23}, we can also extend it into $u \in (0,s/2)$ by decomposing
\begin{align*}
|\mcR_s^tf (x)-\mcR_u^tf(x)|
&\le\sum_{n=0}^\infty| \mcR_{(s/2^n)\wedge u}^tf(x)-\mcR_{(s/2^{n+1})\wedge u}^tf(x)|.
\end{align*}
The same inequality for the case $s=t\le1$ can be obtained by a similar argument.
In the end, the inequality \eqref{eq:singularreconstructioncauchy} holds for any $0<u<s\le t\wedge1$.

\medskip

\noindent
{\bf (2) Convergence as $s\downarrow0$.}\ 
Note that $Q_s\mcR_u^tf=\mcR_u^{t+s}f$ follows from the semigroup property.
By the inequality \eqref{eq:singularreconstructioncauchy}, for any $0<u<s\le t/2$ we have
\begin{align}\label{eq:proof:reconst:cauchy}
\begin{aligned}
&(wv)(x)|\mcR_s^tf(x)-\mcR_u^tf(x)|\\
&\le\int_{\bbR^d}(w^* v^*)(x-y)(wv)(y)|Q_{t/2}(x,y)(\mcR_s^{t/2}f-\mcR_u^{t/2}f)(y)|dy\\
&\lesssim\|\Pi\|_{\gamma,w} \|f\|^{\#}_{\gamma,\eta,v}\, s^{\gamma/\ell}
\int_{\bbR^d}(w^* v^*)(x-y)G_{t/2}(x-y)\varphi (y)^{\eta-\gamma}dy\\
&\lesssim\|\Pi\|_{\gamma,w} \|f\|^{\#}_{\gamma,\eta,v}\, s^{\gamma/\ell}t^{(\eta-\gamma)/\ell}.
\end{aligned}
\end{align}
Since $\gamma>0$, this implies that $\{\mcR_s^tf\}_{0<s\le t/2}$ is Cauchy in $C(wv)$ as $s\downarrow0$. We denote its limit by
$$
\mcR_0^tf:=\lim_{s\downarrow0}\mcR_s^tf.
$$
We also have $Q_s\mcR_0^tf=\mcR_0^{t+s}f$ by taking the limit $u\downarrow0$ in $Q_s\mcR_u^tf=\mcR_u^{t+s}f$.

\medskip

\noindent
{\bf (3) Convergence as $t\downarrow0$.}\ 
Combining the Cauchy property \eqref{eq:proof:reconst:cauchy} and the bound \eqref{eq:proof:reconst:initial} with $s=t/2$, we have
\begin{align*}%\label{eq:besovreconstructuniform}
\begin{aligned}
\|\mcR_0^tf\|_{L^\infty(wv)}
&\le\|\mcR_{t/2}^tf\|_{L^\infty(wv)} +\|\mcR_{t/2}^tf-\mcR_0^tf\|_{L^\infty(wv)}\\
%&\lesssim\|\Pi\|_{\gamma,w} \tri f\tri_{\gamma,\eta,v}^\#\sum_{\alpha\le\gamma}t^{\alpha/\ell}t^{\{(\eta-\alpha)\wedge0\}/\ell}\\
&\lesssim\|\Pi\|_{\gamma,w} \tri f\tri_{\gamma,\eta,v}^\#\, t^{(\eta\wedge\alpha_0)/\ell}.
\end{aligned}
\end{align*}
Since $Q_s\mcR_0^tf=\mcR_0^{t+s}f$, this implies
$$
\sup_{0<t\le 1}\|\mcR_0^tf\|_{C^{\eta\wedge\alpha_0\wedge0,Q}(wv)}
\lesssim\|\Pi\|_{\gamma,w} \tri f\tri_{\gamma,\eta,v}^\#.
$$
From here onward, in exactly the same way as the part (4) of the proof of \cite[Theorem 4.1]{Ho23}, we can show the existence of $\mcR f\in C^{\zeta,Q}(wv)$ with $\zeta=\eta\wedge\alpha_0\wedge0$ which satisfies the bound \eqref{besovreconstruction} and
$$
\lim_{t\downarrow0}\|\mcR f-\mcR_0^t f\|_{C^{\zeta-\varepsilon,Q}(wv)}=0
$$
for any $\varepsilon\in(0,\ell]$. Moreover, we have $Q_t\mcR f=\mcR_0^tf$ by taking the limit $s\downarrow0$ in $Q_t\mcR_0^sf=\mcR_0^{t+s}f$.
We have another bound \eqref{besovreconstructioncharacterize} by letting $u\downarrow0$ and $s=t$ in the inequality \eqref{eq:singularreconstructioncauchy}.

\medskip

\noindent
{\bf (4) Uniqueness.}\
Let $\Lambda,\Lambda'\in C^{\zeta,Q}(wv)$ be reconstructions of $f$ for $M$. 
By the property of reconstruction, $g:=\Lambda-\Lambda'$ satisfies
\begin{equation*}%\label{eq:proof:besovreconstructioncharacterizeuniqueness}
\sup_{x\in\bbR^d}\varphi (x)^{\gamma-\eta}(wv)(x)|Q_tg(x)|\lesssim t^{\gamma/\ell}.
\end{equation*}
Set $\tilde{w}:=\varphi ^{\gamma-\eta}wv$.
By Proposition \ref{proposition:tildew:importantfacts}-\ref{proposition:tildew:importantfacts:embeddingwtildew} and \ref{proposition:tildew:importantfacts:contit}, for any $\varepsilon\in(0,\ell]$ we have
\begin{align*}%\label{decom}
    \|g\|_{C^{\zeta-\varepsilon,Q}(\tilde{w})} 
    &\le \|(Q_t -\id)g\|_{C^{\zeta-\varepsilon,Q}(\tilde{w})} 
    + \|Q_tg\|_{C^{\zeta-\varepsilon,Q}(\tilde{w})}\\
    &\lesssim t^{\varepsilon/\ell}\|g\|_{C^{\zeta,Q}(\tilde{w}) }
    + \|Q_tg\|_{L^\infty(\tilde{w})}\\
    &\lesssim t^{\varepsilon/\ell}\|g\|_{C^{\zeta,Q}(wv) }
    + t^{\gamma/\ell}.
\end{align*}
By taking the limit $t \downarrow 0$, we have $g=0$ in $C^{\zeta-\varepsilon,Q}(\tilde{w})$.
By Proposition \ref{proposition:tildew:importantfacts}-\ref{proposition:tildew:importantfacts:embedding} and \ref{proposition:tildew:importantfacts:embeddingwtildew}, we also have $g=0$ in $C^{\zeta,Q}(wv)$.
\end{proof}

The following result is used in Section \ref{section:application}.

\begin{prop}\label{proposition:reconstsmooth}
In addition to the setting of Theorem \ref{theorem:singularreconstruction}, we assume that the model $M$ is smooth in the sense of Definition \ref{def:model} and
$$
\sup_{x\in\bbR^d}\sup_{\tau\in\bfT_\alpha\setminus\{0\}}w(x)\frac{|(\Pi_x\tau)(x)|}{\|\tau\|_\alpha}<\infty
$$
for any $\alpha\in\bfA$.
Then the reconstruction $\mcR f$ of $f\in\mcD_v^{\gamma,\eta}(\Gamma)^\#$ is realized as a continuous function on $(\bbR\setminus\{0\})\times\bbR^{d-1}$ such that
$$
(\mcR f)(x)=\big(\Pi_xf(x)\big)(x)
$$
for any $x\in(\bbR\setminus\{0\})\times\bbR^{d-1}$.
\end{prop}

\begin{proof}
Set $\Lambda(x)=\big(\Pi_xf(x)\big)(x)$. 
Since $(\Pi_x\tau)(x)=\lim_{t\downarrow0}Q_t(x,\Pi_x\tau)=0$ if $\tau\in\bfT_\alpha$ with $\alpha>0$, we have
\begin{align*}
(wv)(x)|\Lambda(x)|&\le\sum_{\alpha\le0}w(x)\big\|\big(\Pi_x(\cdot)\big)(x)\big\|_{\bfT_\alpha^*}v(x)\|f(x)\|_\alpha\\
&\lesssim\sum_{\alpha\le0}\varphi (x)^{(\eta-\alpha)\wedge0}
\lesssim\varphi (x)^{\eta\wedge0}.
\end{align*}
Since $\eta>-\mfs_1$, we have $\Lambda\in C^{\eta\wedge0,Q}(wv)\subset C^{\zeta,Q}(wv)$ by Corollary \ref{proposition:Gtboundedtildew}. Moreover, since
\begin{align*}
&(wv)(x)|Q_t(x,\Lambda_x)|
=(wv)(x)\bigg|\int_{\bbR^d}Q_t(x,y)\Pi_y\big(f(y)-\Gamma_{yx}f(x)\big)(y)dy\bigg|\\
&\lesssim\sum_{\alpha\le0}\int_{\bbR^d}w^*(x-y)G_t(x-y)
w(y)\big\|\big(\Pi_y(\cdot)\big)(y)\big\|_{\bfT_\alpha^*}
v(x)\|f(y)-\Gamma_{yx}f(x)\|_\alpha
dy\\
&\lesssim\sum_{\alpha\le0}\int_{\bbR^d}(w^*v^*)(x-y)G_t(x-y)\|y-x\|_\mfs^{\gamma-\alpha}\varphi (x,y)^{\eta-\gamma}dy\\
&\lesssim\sum_{\alpha\le0}t^{(\gamma-\alpha)/\ell}\varphi (x)^{\eta-\gamma}
\lesssim t^{\gamma/\ell}\varphi (x)^{\eta-\gamma},
\end{align*}
we have $\lb\Lambda\rb_{\gamma,\eta,wv}<\infty$.
Hence $\mcR f=\Lambda$ by the uniqueness of the reconstruction.
\end{proof}

Combining Theorem \ref{theorem:singularreconstruction} with Proposition \ref{def:variantsnorms}-\ref{def:variantsnorms1}, we have the following result.

\begin{cor}\label{corollary:singularreconstruction}
Assume that $w^2v$ is also $G$-controlled. 
If $\gamma>0$ and $\eta\wedge\alpha_0\in(\gamma-\mfs_1,\gamma]$, then for any $M=(\Pi,\Gamma)\in\scM_{w}(\scT)$ and $f\in\mcD_{v}^{\gamma,\eta}(\Gamma)$, there exists a unique reconstruction $\mcR f\in C^{\eta\wedge\alpha_0\wedge0,Q}(w^2v)$ of $f$ for $M$ and it holds that
\begin{align*}
\|\mcR f\|_{C^{\eta\wedge\alpha_0\wedge0,Q}(w^2v)}&\lesssim\|\Pi\|_{\gamma,w}(1+\|\Gamma\|_{\gamma,w})\tri f\tri_{\gamma,\eta,v},\\
\lb\mcR f\rb_{\gamma,\eta\wedge\alpha_0,w^2v}&\lesssim \|\Pi\|_{\gamma,w}(1+\|\Gamma\|_{\gamma,w})\|f\|_{\gamma,\eta,v}.
\end{align*}
The local Lipschitz estimates similar to the latter part of Theorem \ref{theorem:singularreconstruction} also hold.
\end{cor}

%%%%%%%%%%%%%%%%%%%%%%%%%%%%%%%%%%%%
\section{Multilevel Schauder estimate}\label{section:schauder}
%%%%%%%%%%%%%%%%%%%%%%%%%%%%%%%%%%%%

This section is devoted to the proof of the multilevel Schauder estimate for singular modelled distributions.
After recalling from \cite{Ho23} the basics of regularizing kernels in the first subsection, we prove the multilevel Schauder estimate in the second subsection.

%%%%%%%%%%%%%%%%%%%%%%%%%%%%%%%%%%%%
\subsection{Regularizing kernels}
%%%%%%%%%%%%%%%%%%%%%%%%%%%%%%%%%%%%

We recall from \cite[Section 5.1]{Ho23} the definition of regularizing kernels.

\begin{defi}\label{asmp2}
Let $\bar\beta>0$.
A \emph{$\bar\beta$-regularizing (integral) kernel admissible for $\{Q_t\}_{t>0}$} is a family of continuous functions $\{K_t:\bbR^d\times\bbR^d\to\bbR\}_{t>0}$ which satisfies the following properties for some constants $\delta>0$ and $C_K>0$.
\begin{enumerate}
\renewcommand{\theenumi}{(\roman{enumi})}
\renewcommand{\labelenumi}{(\roman{enumi})}
\item\label{asmp2:convolutionwithQ}
(Convolution with $Q$)
For any $0<s<t$ and $x,y\in\bbR^d$,
$$
\int_{\bbR^d}K_{t-s}(x,z)Q_s(z,y)dz=K_t(x,y).
$$
\item\label{asmp2:gauss}
(Upper estimate)
For any ${\bf k}\in\bbN^d$ with $|{\bf k}|_\mfs<\delta$, the ${\bf k}$-th partial derivative of $K_t(x,y)$ with respect to $x$ exists, and we have for any $t>0$ and $x,y\in\bbR^d$,
$$
|\partial_x^{\bf k}K_t(x,y)|\le C_Kt^{(\bar\beta-|{\bf k}|_\mfs)/\ell-1}G_t(x-y).
$$
\item\label{asmp2:holder}
(H\"older continuity)
For any ${\bf k}\in\bbN^d$ with $|{\bf k}|_\mfs<\delta$, any $t>0$ and $x,y,h\in\bbR^d$ with $\|h\|_{\mfs}\le t^{1/\ell}$,
\begin{align*}
&\bigg|\partial_x^{\bf k}K_t(x+h,y)-\sum_{|{\bf l}|_\mfs<\delta-|{\bf k}|_\mfs}\frac{h^{\bf l}}{{\bf l}!}\partial_x^{{\bf k}+{\bf l}}K_t(x,y)\bigg|\\
& \le C_K\|h\|_\mfs^{\delta-|{\bf k}|_\mfs}\,  t^{(\bar\beta-\delta)/\ell-1}G_t(x-y).
\end{align*}
\end{enumerate}
\end{defi}

%\begin{remark}\label{asmp2'}
%The property \ref{asmp2:holder} still holds if we replace $\delta$ with arbitrary $\varepsilon\in(0,\delta)$.
%To see this, we have only to decompose
%\begin{align*}
%&\partial_x^{\bf k}K_t(x+h,y)-\sum_{|{\bf l}|_\mfs<\varepsilon-|{\bf k}|_\mfs}\frac{h^{\bf l}}{{\bf l}!}\partial_x^{{\bf k}+{\bf l}}K_t(x,y)\\
%&=\bigg(\partial_x^{\bf k}K_t(x+h,y)-\sum_{|{\bf l}|_\mfs<\delta-|{\bf k}|_\mfs}\frac{h^{\bf l}}{{\bf l}!}\partial_x^{{\bf k}+{\bf l}}K_t(x,y)\bigg)+\sum_{\varepsilon-|{\bf k}|_{\mfs}\le|{\bf l}|_{\mfs}<\delta-|{\bf k}|_{\mfs}}\frac{h^{\bf l}}{{\bf l}!}\partial_x^{{\bf k}+{\bf l}}K_t(x,y)
%\end{align*}
%and use properties \ref{asmp2:gauss}, \ref{asmp2:holder}, and the condition $\|h\|_{\mfs}\le t^{1/\ell}$.
%\end{remark}

We fix a $\bar\beta$-regularizing kernel $\{K_t\}_{t>0}$ throughout this section.
For any $f\in L^\infty(w)$ with a $G$-controlled weight $w$ and any $|{\bf k}|_{\mfs}<\delta$, we define
$$
(\partial^{\bf k}K_tf)(x):=:\partial^{\bf k}K_t(x,f):=\int_{\bbR^d}\partial_x^{\bf k}K_t(x,y)f(y)dy.
$$
Moreover, we write $\partial^{\bf k}Kf:=\int_0^1\partial^{\bf k}K_tf dt$ if the integral makes sense.

\begin{lem}\label{lemma:whyregularize}
Let $w$ and $v$ be $G$-controlled weights such that $w^2$ and $wv$ are also $G$-controlled.
Let $\scT=(\bfA,\bfT,\bfG)$ be a regularity structure
and let $M=(\Pi,\Gamma)\in\scM_w(\scT)$.
\begin{enumerate}
\renewcommand{\theenumi}{(\roman{enumi})}
\renewcommand{\labelenumi}{(\roman{enumi})}
\item\label{lemma:whyregularize1}
\cite[Lemma 5.4]{Ho23}
For any $\alpha\le0$, $|{\bf k}|_{\mfs}<\delta$, and $f\in L^\infty(w)$, we have
$$
\|\partial^{\bf k}K_tf\|_{L^\infty(w)}\lesssim C_K\,t^{(\alpha+\bar\beta-|{\bf k}|_\mfs)/\ell-1}\|f\|_{C^{\alpha,Q}(w)},
$$
where the implicit proportional constant depends only on $G$ and $w$.
Consequently, if $|{\bf k}|_\mfs<(\alpha+\bar\beta)\wedge\delta$, the integral $\partial^{\bf k}Kf:=\int_0^1\partial^{\bf k}K_tfdt$ converges in $C(w)$.

\item\label{lemma:KregularizePi}
\cite[Lemma 5.6]{Ho23}
For any $\alpha<\gamma$, $\tau\in\bfT_\alpha$, $|{\bf k}|_{\mfs}<\delta$, and $t\in(0,1]$, we have
$$
\|\partial^{\bf k}K_t(x,{\Pi}_x\tau)\|_{L_x^\infty(w^2)}
\lesssim C_K\,t^{(\alpha+\bar\beta-|{\bf k}|_{\mfs})/\ell-1}\|\Pi\|_{\gamma,w}(1+\|\Gamma\|_{\gamma,w})\|\tau\|_\alpha,
$$
where the implicit proportional constant depends only on $G,w$, and $\bfA$.
Consequently, if $|{\bf k}|_\mfs<(\alpha+\bar\beta)\wedge\delta$, the integral $\partial^{\bf k}K(x,\Pi_x\tau):=\int_0^1\partial^{\bf k}K_t(x,\Pi_x\tau)dt$ converges for any $x\in\bbR^d$.

\item\label{lemma:KregularizeRf}
Let $\gamma\in\bbR$, $\eta\in(\gamma-\mfs_1,\gamma]$, and $\zeta\le0$.
For any $f\in\mcD_v^{\gamma,\eta}(\Gamma)^\#$ and its reconstruction $\Lambda\in C^{\zeta,Q}(wv)$, $|{\bf k}|_{\mfs}<\delta$, and $t\in(0,1]$, we have
\begin{align*}%\label{3.8}
& (wv)(x)|\partial^{\bf k} K_t(x,\Lambda_x)|\\
&\lesssim C_K\,t^{(\gamma+\bar\beta-|{\bf k}|_\mfs)/\ell-1}\, \varphi (x)^{\eta-\gamma}
     \big(\lb\Lambda\rb_{\gamma,\eta,wv}+\|\Pi\|_{\gamma,w}\|f\|_{\gamma,\eta,v}^\#\big),
\end{align*}
where the implicit proportional constant depends only on $G,w,v$, and $\bfA$.
Consequently, if $|{\bf k}|_\mfs<(\gamma+\bar\beta)\wedge\delta$, the integral $\partial^{\bf k}K(x,\Lambda_x):=\int_0^1\partial^{\bf k}K_t(x,\Lambda_x)dt$ converges for any $x\in(\bbR\setminus\{0\})\times\bbR^{d-1}$.
\end{enumerate}
\end{lem}

\begin{proof}
We prove only \ref{lemma:KregularizeRf}.
By Definition \ref{asmp2}-\ref{asmp2:convolutionwithQ}, we can decompose
\begin{align*}
|\partial^{\bf k}K_t(x,\Lambda_x)|
&\le\bigg|\int_{\bbR^d}\partial^{\bf k}K_{t/2}(x,y)Q_{t/2}(y,\Lambda_y)dy\bigg|\\
& \quad +\bigg|\int_{\bbR^d}\partial^{\bf k}K_{t/2}(x,y)Q_{t/2}\big(y,\Pi_yf(y)-\Pi_xf(x)\big)dy\bigg|.
\end{align*}
For the first term, by Definition \ref{asmp2}-\ref{asmp2:gauss} and by the property of reconstruction, we have
\begin{align*}
& (wv)(x)\bigg|\int_{\bbR^d}\partial^{\bf k}K_{t/2}(x,y)Q_{t/2}(y,\Lambda_y)dy\bigg|\\
& \lesssim C_K\,t^{(\bar\beta-|{\bf k}|_\mfs)/\ell-1}\int_{\bbR^d} (w^*v^*)(x-y)G_{t/2}(x-y) (wv)(y)|Q_{t/2}(y,\Lambda_y)| dy\\
& \lesssim C_K\,t^{(\gamma+\bar\beta-|{\bf k}|_\mfs)/\ell-1} \lb\Lambda\rb_{\gamma,\eta,wv}
\int_{\bbR^d} \varphi (y)^{\eta-\gamma} (w^*v^*)(x-y) G_{t/2}(x-y)  dy\\ 
&\lesssim C_K\,t^{(\gamma+\bar\beta-|{\bf k}|_\mfs)/\ell-1}\,\varphi (x)^{\eta-\gamma}\,
\lb\Lambda\rb_{\gamma,\eta,wv}.
\end{align*}
For the second term, by using the inequality \eqref{eq:proof:reconst:coherence} obtained in the proof of Theorem \ref{theorem:singularreconstruction} with $x$ and $y$ swapped, we have
\begin{align*}
& (wv)(x)\bigg|\int_{\bbR^d}\partial^{\bf k}K_{t/2}(x,y)Q_{t/2}\big(y,\Pi_yf(y)-\Pi_xf(x)\big)dy\bigg|\\
& \lesssim  C_K\,t^{(\bar\beta-|{\bf k}|_\mfs)/\ell-1}\int_{\bbR^d} G_{t/2}(x-y) (wv)(x)\big|Q_{t/2}\big(y,\Pi_yf(y)-\Pi_xf(x)\big)\big| dy\\
&\lesssim
C_K\|\Pi\|_{\gamma,w} \|f\|_{\gamma,\eta,v}^\#
\sum_{\alpha<\gamma} t^{(\alpha+\bar\beta-|{\bf k}|_\mfs)/\ell-1}
\int_{\bbR^d}\varphi (x,y)^{\eta-\gamma}\|y-x\|_\mfs^{\gamma-\alpha}\\
&\hspace{200pt}\times(w^*v^*)(x-y)G_{t/2}(x-y)dy 
\\
&\lesssim C_K\,t^{(\gamma+\bar\beta-|{\bf k}|_\mfs)/\ell-1}\,\varphi (x)^{\eta-\gamma}\,\|\Pi\|_{\gamma,w}\|f\|_{\gamma,\eta,v}^\#.
\end{align*}
\end{proof}

%%%%%%%%%%%%%%%%%%%%%%%%%%%%%%%%%%%%
\subsection{Compatible models and multilevel Schauder estimate}
%%%%%%%%%%%%%%%%%%%%%%%%%%%%%%%%%%%%

We recall from \cite[Section 5.2]{Ho23} the notions of abstract integrations and compatible models.
Hereafter, we use the \emph{polynomial structure} generated by dummy variables $X_1,\dots,X_d$ as in \cite[Section 2]{Hai14}.

\begin{defi}\label{defTbarT}
Let $\bar\scT=(\bar{\bfA},\bar{\bfT},\bar{\bfG})$ be a regularity structure satisfying the following properties.
\begin{itemize}
\item[(1)]
$\bbN[\mfs]\subset \bar{\bfA}$.
\item[(2)]
For each $\alpha\in\bbN[\mfs]$, the space $\bar{\bfT}_{\alpha}$ contains all $X^{\bf k}:=\prod_{i=1}^dX_i^{k_i}$ with $|{\bf k}|_{\mfs}=\alpha$.
\item[(3)]
The subspace $\spa\{X^{\bf k}\}_{{\bf k}\in\bbN^d}$ of $\bar\bfT$ is closed under $\bar{\bfG}$-actions.
\end{itemize}
Let $\scT=(\bfA,\bfT,\bfG)$ be another regularity structure.
A continuous linear operator $\mcI:\bfT\to\bar{\bfT}$ is called an \emph{abstract integration} of order $\beta\in(0,\bar\beta]$ if 
$$
\mcI:\bfT_\alpha\to\bar{\bfT}_{\alpha+\beta}
$$
for any $\alpha\in \bfA$.
For a fixed $G$-controlled weight $w$, we say that the pair $(M,\bar{M})$ of two models $M=(\Pi,\Gamma)\in\scM_w(\scT)$ and $\bar{M}=(\bar{\Pi},\bar{\Gamma})\in\scM_w(\bar{\scT})$ is \emph{compatible for $\mcI$} if it satisfies the following properties.
\begin{enumerate}
\renewcommand{\theenumi}{(\roman{enumi})}
\renewcommand{\labelenumi}{(\roman{enumi})}
\item\label{Iadmissible2}
For any ${\bf k}\in\bbN^d$,
$$
(\bar{\Pi}_xX^{\bf k})(\cdot)=(\cdot-x)^{\bf k},\qquad
\bar\Gamma_{yx} X^{\bf k}=\sum_{{\bf l}\le {\bf k}}\binom{\bf k}{\bf l}(y-x)^{\bf l} X^{{\bf k}-{\bf l}}.
$$
\item\label{Iadmissible3}
For each $x\in\bbR^d$, we define the linear map $\mcJ(x):\bfT_{<\delta-\beta}\to\spa\{X^{\bf k}\}_{|{\bf k}|_\mfs<\delta}\subset\bar{\bfT}$ by setting
\begin{align}\label{eq:Jx}
\mcJ(x)\tau=\sum_{|{\bf k}|_{\mfs}<\alpha+\beta}\frac{X^{\bf k}}{{\bf k}!}\partial^{\bf k}K(x,\Pi_x\tau)
\end{align}
for any $\alpha\in \bfA$ such that $\alpha+\beta<\delta$ and $\tau\in \bfT_{\alpha}$.
Then for any $\tau\in\bfT_{<\delta-\beta}$,
\begin{align*}%\label{ex:compatiblemmaodel}
\bar{\Gamma}_{yx}\big(\mcI+\mcJ(x)\big)\tau=\big(\mcI+\mcJ(y)\big)\Gamma_{yx}\tau.
\end{align*}
\end{enumerate}
In addition, if the regularity $\alpha_0$ of $\scT$ is greater than $-\bar\beta$ and
\begin{align}\label{ex:admissiblemmaodel}
(\bar{\Pi}_x\mcI\tau)(\cdot)=K(\cdot,\Pi_x\tau)-\sum_{|{\bf k}|_{\mfs}<\alpha+\beta}\frac{(\cdot-x)^{\bf k}}{{\bf k}!}\partial^{\bf k}K(x,\Pi_x\tau)
\end{align}
for any $\tau\in\bfT_\alpha$ with $\alpha+\beta<\delta$, then we say that the pair $(M,\bar{M})$ is \emph{$K$-admissible for $\mcI$}.
\end{defi}

In \eqref{eq:Jx} and \eqref{ex:admissiblemmaodel}, 
the function $K(\cdot,\Pi_x\tau)$ and the coefficients $\partial^{\bf k}K(x,\Pi_x\tau)$ are well-defined by Lemma \ref{lemma:whyregularize}.
The following theorem is the second main result of this paper.

\begin{thm}\label{theorem:singularMSE}
Let $\scT$ and $\bar{\scT}$ be regularity structures satisfying the setting of Definition \ref{defTbarT} and let $\mcI:\bfT\to\bar\bfT$ be an abstract integration of order $\beta\in(0,\bar\beta]$.
Let $w$ and $v$ be $G$-controlled weights such that $w^2v$ is also $G$-controlled.
Given $(\Pi,\Gamma)\in\scM_{w}(\scT)$, $f\in\mcD_v^{\gamma,\eta}(\Gamma)^\#$ with $\gamma+\bar\beta<\delta$ and $\eta\in(\gamma-\mfs_1,\gamma]$, and its reconstruction $\Lambda\in C^{\zeta,Q}(wv)$, we define the functions
$$
\mcN(x;f,\Lambda)=\sum_{|{\bf k}|_{\mfs}<\gamma+\beta}\frac{X^{\bf k}}{{\bf k}!}\partial^{\bf k}K(x,\Lambda_x)%\qquad (x\in(\bbR\setminus\{0\})\times\bbR^{d-1})
$$
and
$$
\mcK f(x):=\mcI f(x)+\mcJ(x)f(x)+\mcN(x;f,\Lambda)%.\qquad (x\in(\bbR\setminus\{0\})\times\bbR^{d-1})
$$
for $x\in(\bbR\setminus\{0\})\times\bbR^{d-1}$.
We assume $\zeta\le\eta\wedge\alpha_0$ and either of the following conditions.
\begin{itemize}
\item[(1)]
$\beta<\bar\beta$.
\item[(2)]
$\beta=\bar\beta$ and $\big\{\alpha+\bar\beta\, ;\, \alpha\in\bfA\cup\{\gamma,\zeta\}\big\}\cap\bbN[\mfs]=\emptyset$.
\end{itemize}
Then for any compatible pair of models $\big(M=(\Pi,\Gamma),\bar{M}=(\bar{\Pi},\bar{\Gamma})\big)\in\scM_w(\scT)\times\scM_w(\bar{\scT})$ and any singular modelled distribution $f\in\mcD_{v}^{\gamma,\eta}(\Gamma)^\#$, the function $\mcK f$ belongs to $\mcD_{w^2v}^{\gamma+\beta,\zeta+\beta}(\bar\Gamma)^\#$, and we have
\begin{align}
\label{MSbesov1}
&
\begin{aligned}
\lp \mcK f\rp_{\gamma+\beta,\zeta+\beta,w^2v}
&\lesssim
\|\mcI\|\lp f\rp_{\gamma,\eta,v}
+C_K\big\{\|\Pi\|_{\gamma,w}(1+\|\Gamma\|_{\gamma,w})\tri f\tri_{\gamma,\eta,v}^\#\\
&\hspace{130pt}+\|\Lambda\|_{C^{\zeta,Q}(wv)}+\lb\Lambda\rb_{\gamma,\eta,wv}\big\},
\end{aligned}
\\
\label{MSbesov2}
&\|\mcK f\|_{\gamma+\beta,\zeta+\beta,w^2v}^\#
\lesssim\|\mcI\| \|f\|_{\gamma,\eta,v}^\#
+C_K\big\{
\|\Pi\|_{\gamma,w} (1+\|\Gamma\|_{\gamma,w})\| f\|_{\gamma,\eta,v}^\#
+\lb\Lambda\rb_{\gamma,\eta,wv}
\big\},
\end{align}
where $\|\mcI\|$ is the operator norm from $\bfT_{<\gamma}$ to $\bar{\bfT}_{<\gamma+\beta}$, and the implicit proportional constant depends only on $G,w,v,\gamma,\eta$, and $\bfA$.
Moreover, there is a quadratic function $C_R>0$ of $R>0$ such that
\begin{align*}
\tri\mcK f^{(1)};\mcK f^{(2)}\tri_{\gamma+\beta,\zeta+\beta,w^2v}^\#
\le C_R\Big(
\tri M^{(1)};M^{(2)}\tri_{\gamma,w}+\tri f^{(1)};f^{(2)}\tri_{\gamma,\eta,v}^\#
\Big),
\end{align*}
for any $M^{(i)}=(\Pi^{(i)},\Gamma^{(i)})\in\scM_{w}(\scT)$ and $\bar{M}^{(i)}=(\bar\Pi^{(i)},\bar\Gamma^{(i)})\in\scM_w(\bar\scT)$ such that $(M^{(i)},\bar{M}^{(i)})$ is compatible and any $f^{(i)}\in\mcD_{v}^{\gamma,\eta}(\Gamma^{(i)})$ with $i\in\{1,2\}$ such that $\tri M^{(i)}\tri_{\gamma,w}\le R$ and $\tri f^{(i)}\tri_{\gamma,\eta,v}^\#\le R$.
\end{thm}

\begin{proof}
The proof is carried out by a method similar to that of \cite[Theorem 5.12]{Ho23}, but we have to prove \eqref{MSbesov1} more carefully than \cite{Ho23}.
For the $\mcI$ term, by the continuity of $\mcI$ we immediately have
\begin{align*}
v(x)\|\mcI f(x)\|_\alpha
\le v(x)\|\mcI\|\|f(x)\|_{\alpha-\beta}
\le\|\mcI\|\lp f\rp_{\gamma,\eta,v}\,\varphi (x)^{(\eta+\beta-\alpha)\wedge0}
\end{align*}
for any $\alpha<\gamma+\beta$.
For the $\mcJ$ and $\mcN$ terms, we decompose
$$
\mcJ(x)f(x)+\mcN(x,f;\Lambda)=\sum_{|{\bf k}|_\mfs<\gamma+\beta}\frac{X^{\bf k}}{{\bf k}!}\mcA^{\bf k}(x),
$$
where
$$
\mcA^{\bf k}(x)
=\sum_{\alpha\in[\alpha_0,\gamma),\,|{\bf k}|_\mfs<\alpha+\beta}\partial^{\bf k}K\big(x,\Pi_xP_\alpha f(x)\big)+\partial^{\bf k}K(x,\Lambda_x).
$$
We further define the decomposition $\mcA^{\bf k}(x)=\int_0^1\mcA^{\bf k}_t(x)dt$ according to the integral form $K=\int_0^1 K_tdt$, where $\mcA_t^{\bf k}$ is defined in the same way as $\mcA^{\bf k}$ with $K$ replaced by $K_t$.
By using Lemma \ref{lemma:whyregularize}-\ref{lemma:KregularizePi} for $\partial^{\bf k}K_t\big(x,\Pi_xP_\alpha f(x)\big)$ and \ref{lemma:KregularizeRf} for $\partial^{\bf k}K_t(x,\Lambda_x)$, we have
\begin{align*}
(w^2v)(x)|\mcA_t^{\bf k}(x)|
\lesssim L_1
\sum_{\alpha\in[\alpha_0,\gamma],\,|{\bf k}|_\mfs<\alpha+\beta}
\varphi (x)^{(\eta-\alpha)\wedge0}\,t^{(\alpha+\bar\beta-|{\bf k}|_\mfs)/\ell-1}
\end{align*}
where $L_1:=C_K\big\{\|\Pi\|_{\gamma,w}(1+\|\Gamma\|_{\gamma,w})\tri f\tri_{\gamma,\eta,v}^\#+\lb\Lambda\rb_{\gamma,\eta,wv}\big\}$.
Since all powers of $t$ above are greater than $-1$, we have
\begin{align*}
(w^2v)(x)\int_0^{\varphi (x)^\ell}|\mcA_t^{\bf k}(x)|dt
&\lesssim L_1\sum_{\alpha\in[\alpha_0,\gamma],\,|{\bf k}|_\mfs<\alpha+\beta}
\varphi (x)^{\eta\wedge\alpha+\bar\beta-|{\bf k}|_\mfs}\\
&\lesssim L_1\,\varphi (x)^{(\eta\wedge\alpha_0+\beta-|{\bf k}|_\mfs)\wedge0}.
\end{align*}
For the integral over $\varphi (x)^\ell<t\le1$, we use another decomposition
$$
\mcA_t^{\bf k}(x)
=\partial^{\bf k}K_t\big(x,\Lambda)-
\sum_{\alpha\in[\alpha_0,\gamma),\,|{\bf k}|_\mfs\ge\alpha+\beta}\partial^{\bf k}K_t\big(x,\Pi_xP_\alpha f(x)\big)
$$
and consider the two terms in the right hand side separately.
For the first term, by the assumption that $\Lambda\in C^{\zeta,Q}(wv)$ and by Lemma \ref{lemma:whyregularize}-\ref{lemma:whyregularize1}, we have
$$
(wv)(x)|\partial^{\bf k}K_t\big(x,\Lambda)|
\lesssim C_K\|\Lambda\|_{C^{\zeta,Q}(wv)}\, t^{(\zeta+\bar\beta-|{\bf k}|_\mfs)/\ell-1}.
$$
If $\zeta+\bar\beta-|{\bf k}|_\mfs\neq0$, we have
$$
\int_{\varphi (x)^\ell}^1t^{(\zeta+\bar\beta-|{\bf k}|_\mfs)/\ell-1}dt
\lesssim\varphi (x)^{(\zeta+\bar\beta-|{\bf k}|_\mfs)\wedge0}
\lesssim\varphi (x)^{(\zeta+\beta-|{\bf k}|_\mfs)\wedge0}.
$$
Otherwise, since $\zeta+\beta-|{\bf k}|_\mfs<\zeta+\bar\beta-|{\bf k}|_\mfs=0$ by assumption we have
$$
\int_{\varphi (x)^\ell}^1t^{(\zeta+\bar\beta-|{\bf k}|_\mfs)/\ell-1}dt
\lesssim\int_{\varphi (x)^\ell}^1t^{(\zeta+\beta-|{\bf k}|_\mfs)/\ell-1}dt
\lesssim\varphi (x)^{\zeta+\beta-|{\bf k}|_\mfs}.
$$
In either case, we obtain the desired estimate.
For the remaining term, by Lemma \ref{lemma:whyregularize}-\ref{lemma:KregularizePi} we have 
\begin{align*}
&(w^2v)(x)\sum_{\alpha\in[\alpha_0,\gamma),\,|{\bf k}|_\mfs\ge\alpha+\beta}
\big|\partial^{\bf k}K_t\big(x,\Pi_xP_\alpha f(x)\big)\big|\\
&\lesssim L_2
\sum_{\alpha\in[\alpha_0,\gamma),\,|{\bf k}|_\mfs\ge\alpha+\beta}
\varphi (x)^{(\eta-\alpha)\wedge0}\,t^{(\alpha+\bar\beta-|{\bf k}|_\mfs)/\ell-1},
\end{align*}
where $L_2:=C_K\|\Pi\|_{\gamma,w}(1+\|\Gamma\|_{\gamma,w})\lp f\rp_{\gamma,\eta,v}$.
For $\alpha$ such that $|{\bf k}|_\mfs>\alpha+\beta$, we easily have
\begin{align*}
\varphi (x)^{(\eta-\alpha)\wedge0}\int_{\varphi (x)^\ell}^1t^{(\alpha+\bar\beta-|{\bf k}|_\mfs)/\ell-1}dt
&\lesssim
\varphi (x)^{(\eta-\alpha)\wedge0}\int_{\varphi (x)^\ell}^1t^{(\alpha+\beta-|{\bf k}|_\mfs)/\ell-1}dt\\
&\lesssim\varphi (x)^{\eta\wedge\alpha+\beta-|{\bf k}|_\mfs}.
\end{align*}
If there exists $\alpha$ such that $|{\bf k}|_\mfs=\alpha+\beta$, then since 
$0=\alpha+\beta-|{\bf k}|_\mfs<\alpha+\bar\beta-|{\bf k}|_\mfs$ by assumption, we have
\begin{align*}
\varphi (x)^{(\eta-\alpha)\wedge0}\int_{\varphi (x)^\ell}^1t^{(\alpha+\bar\beta-|{\bf k}|_\mfs)/\ell-1}dt
\lesssim
\varphi (x)^{(\eta-\alpha)\wedge0}
=\varphi (x)^{\eta\wedge\alpha+\beta-|{\bf k}|_\mfs}.
\end{align*}
Consequently, we obtain
\begin{align*}
(w^2v)(x)\int_{\varphi (x)^\ell}^1|\mcA_t^{\bf k}(x)|dt
\lesssim\{C_K\|\Lambda\|_{C^{\zeta,Q}(wv)}+L_2\}
\varphi (x)^{(\zeta+\beta-|{\bf k}|_\mfs)\wedge0}.
\end{align*}

The proof of \eqref{MSbesov2} is completely the same as that of \cite[Theorem 5.12]{Ho23} except the existence of the factor $\varphi (x,y)^{\eta-\gamma}$.
\end{proof}

The following theorem is obtained similarly to \cite[Theorem 5.13]{Ho23}, so we omit the proof.

\begin{thm}\label{theorem:KR=RK}
In addition to the setting of Theorem \ref{theorem:singularMSE}, we assume that $\zeta+\bar\beta>0$ and that $(M,\bar{M})$ is $K$-admissible for $\mcI$. Then $K\Lambda\in C(wv)$ is a reconstruction of $\mcK f\in\mcD_{w^2v}^{\gamma+\beta,\zeta+\beta}(\bar\Gamma)^\#$ and
$$
\lb K\Lambda\rb_{\gamma+\beta,\zeta+\beta,w^2v}\lesssim C_K\big(\lb\Lambda\rb_{\gamma,\eta,wv}+\|\Pi\|_{\gamma,w}\|f\|_{\gamma,\eta,v}^\#\big).
$$
A similar local Lipschitz estimate to the latter part of Theorem \ref{theorem:singularMSE} also holds.
\end{thm}

Combining Theorem \ref{theorem:singularMSE} with Proposition \ref{def:variantsnorms}-\ref{def:variantsnorms1}, we have the following result.

\begin{cor}\label{corollary:singularMSE}
In addition to the setting of Theorem \ref{theorem:singularMSE}, assume that $w^3v$ is $G$-controlled and that $\alpha_0>\gamma-\mfs_1$.
Then for any compatible pair of models $\big(M=(\Pi,\Gamma),\bar{M}=(\bar{\Pi},\bar{\Gamma})\big)\in\scM_w(\scT)\times\scM_w(\bar{\scT})$ and any singular modelled distribution $f\in\mcD_{v}^{\gamma,\eta}(\Gamma)$, the function $\mcK f$ belongs to $\mcD_{w^3v}^{\gamma+\beta,\zeta+\beta}(\bar\Gamma)$, and we have
\begin{align*}
&
\begin{aligned}
\lp \mcK f\rp_{\gamma+\beta,\zeta+\beta,w^3v}
&\lesssim
\|\mcI\|\lp f\rp_{\gamma,\eta,v}
+C_K\big\{\|\Pi\|_{\gamma,w}(1+\|\Gamma\|_{\gamma,w})^2\tri f\tri_{\gamma,\eta,v}\\
&\hspace{130pt}+\|\Lambda\|_{C^{\zeta,Q}(wv)}+\lb\Lambda\rb_{\gamma,\eta,wv}\big\},
\end{aligned}
\\
&\|\mcK f\|_{\gamma+\beta,\zeta+\beta,w^3v}
\lesssim\|\mcI\| \big\{\|\Gamma\|_{\gamma,w}\lp f\rp_{\gamma,\eta,v}+\| f\|_{\gamma,\eta,v}\big\}\\
&\hspace{100pt}+C_K\big\{
\|\Pi\|_{\gamma,w} (1+\|\Gamma\|_{\gamma,w})^2\| f\|_{\gamma,\eta,v}
+\lb\Lambda\rb_{\gamma,\eta,wv}
\big\}.
\end{align*}
A similar local Lipschitz estimate to the latter part of Theorem \ref{theorem:singularMSE} also holds.
\end{cor}

%%%%%%%%%%%%%%%%%%%%%%%%%%%%%%%%%%%%%
%\section{Multiplication of modelled distributions}\label{section:multiplication}
%%%%%%%%%%%%%%%%%%%%%%%%%%%%%%%%%%%%%
%
%
%
%\begin{theorem}
%\end{theorem}
%
%\begin{proof}
%The bound of $\lp f\rp_{\gamma,\eta,v}$ is easily obtained as follows.
%\begin{align*}
%(v_1v_2)(x)\|f(x)\|_\alpha
%&\lesssim\sum_{\alpha_1+\alpha_2=\alpha}v_1(x)\|f_1(x)\|_{\alpha_1}v_2(x)\|f_2(x)\|_{\alpha_2}\\
%&\le\lp f_1\rp_{\gamma_1,\eta_1,v_1}\lp f_2\rp_{\gamma_2,\eta_2,v_2}
%\sum_{\alpha_1+\alpha_2=\alpha}\varphi (x)^{(\eta_1-\alpha_1)\wedge0+(\eta_2-\alpha_2)\wedge0}\\
%&\lesssim\lp f_1\rp_{\gamma_1,\eta_1,v_1}\lp f_2\rp_{\gamma_2,\eta_2,v_2}\,\varphi (x)^{(\eta-\alpha)\wedge0}.
%\end{align*}
%To show the bound of $\| f\|_{\gamma,\eta,v}$, by following the proof of \cite[theorem 4.7]{Hai14}, we decompose
%\begin{align*}
%\|f(y)-\Gamma_{yx}f(x)\|_\alpha
%\le
%\end{align*}
%\end{proof}

%%%%%%%%%%%%%%%%%%%%%%%%%%%%%%%%%%%%
\section{Parabolic Anderson model}\label{section:application}
%%%%%%%%%%%%%%%%%%%%%%%%%%%%%%%%%%%%

In this section, we study the parabolic Anderson model (PAM)
\begin{equation}\label{5:eq:PAM}
\big(\partial_1-a(x')\Delta+c\big)u(x)=b\big(u(x)\big)\xi(x')
\qquad(x\in(0,\infty)\times\bbT^2)
\end{equation}
with a spatial white noise $\xi$ defined on a probability space $(\Omega,\mcF,\bbP)$. Recall that $x_1$ in $x=(x_1,x_2,x_3)$ denotes the temporal variable and $x'=(x_2,x_3)$ denotes the spatial variables.
Throughout this section, we fix the function $b:\bbR\to\bbR$ in the class $C_b^3$,
and the function $a:\bbT^2\to\bbR$ which is $\alpha$-H\"older continuous for some $\alpha\in(0,1)$ and satisfies
$$
C_1\le a(x')\le C_2\qquad(x'\in\bbT^2)
$$
for some constants $0<C_1<C_2$.
The constant $c>0$ in the left hand side of \eqref{5:eq:PAM} is fixed later (see Proposition \ref{5:proposition:QtGtype} and \ref{5:proposition:convolutionops}). %Although we consider a parabolic operator with a function coefficient, this equation has the same problem as the constant coefficient case, namely, it contains an ill-defined product $f(u)\xi$.
We prove the renormalizability of \eqref{5:eq:PAM} in Section \ref{5:sec:renorPAM}.
We fix $\alpha\in(0,1)$, $d=3$, $\mfs=(2,1,1)$, and $\ell=4$ throughout this section.

%%%%%%%%%%%%%%%%%%%%%%%%%%%%%%%%%%%%
\subsection{Preliminaries}
%%%%%%%%%%%%%%%%%%%%%%%%%%%%%%%%%%%%

We denote by $e_1=(1,0,0)$, $e_2=(0,1,0)$, and $e_3=(0,0,1)$ the canonical basis vectors of $\bbR^3$.
We define $C_b(\bbR\times\bbT^2)$ as the set of all bounded continuous functions $f:\bbR^3\to\bbR$ such that 
$$
f(x+e_i)=f(x)
$$
for any $x\in\bbR^3$ and $i\in\{2,3\}$.
For any $\beta>0$, we define $C_{\mfs}^\beta(\bbR\times\bbT^2)$ as the set of all elements $f\in C_b(\bbR\times\bbT^2)$ such that $\partial_x^kf\in C_b(\bbR\times\bbT^2)$ for any $|k|_{\mfs}<\beta$, and if $|k|_{\mfs}<\beta\le|k|_{\mfs}+\mfs_i$, we have
$$
|\partial^kf(x+he_i)-\partial^kf(x)|\lesssim|h|^{(\beta-|k|_\mfs)/\mfs_i}
$$
for any $x\in\bbR^3$ and $h\in\bbR$.

We denote by $P_{x_1}(x',y')$ the fundamental solution of the parabolic operator $\partial_1-a\Delta+c$.
Moreover, we introduce the anisotropic elliptic operator
$$
\mcL:=\big(\partial_1-a(x')\Delta\big)(\partial_1+\Delta)
$$
on $\bbR^3$ and denote by $Q_t(x,y)$ the fundamental solution of $\partial_t-\mcL+c$ with an additional variable $t>0$.
We recall from \cite[Appendix A]{BHK22} some properties of $P_{x_1}(x',y')$ and $Q_t(x,y)$.

\begin{prop}
[{\cite[Theorem 57]{BHK22}}]\label{5:proposition:QtGtype}
For any $C>0$, we define the function $G^{(C)}$ on $\bbR^3$ by
$$
G^{(C)}(x)=\exp\big\{-C\big(|x_1|^2+|x_2|^{4/3}+|x_3|^{4/3}\big)\big\}.
$$
For sufficiently large $c>0$, $\{Q_t\}_{t>0}$ is a $G^{(C)}$-type semigroup for some constant $C>0$, in the sense of Definition \ref{asmp1}. 
\end{prop}

In what follows, we fix $C>0$ and write $G=G^{(C)}$.
For any $G$-controlled weight $w$ and any $\zeta\le0$, we can define the Besov space $C^{\zeta,Q}(w)$ in the sense of Definition \ref{def:BesovassociatedQ}.
We denote by $C^{\zeta,Q}(\bbR\times\bbT^2)$ the closure of $C_b(\bbR\times\bbT^2)$ in the space $C^{\zeta,Q}(1)$ with the flat weight $w=1$.

\begin{prop}\label{5:proposition:convolutionops}
For sufficiently large $c>0$, we have the following.
\begin{enumerate}
\renewcommand{\theenumi}{(\roman{enumi})}
\renewcommand{\labelenumi}{(\roman{enumi})}
\item\cite[Theorems 61 and Proposition 62]{BHK22}
Let $\beta\in(0,\alpha)$.
For any $g\in C_\mfs^\beta(\bbR\times\bbT^2)$, we can define the function on $\bbR\times\bbT^2$ by
\begin{align*}
\big((\partial_1-a\Delta+c)^{-1}g\big)(x) := 
\int_{(-\infty,x_1]\times\bbR^2}P_{x_1-y_1}(x',y')g(y)dy.
\end{align*}
Then $h=(\partial_1-a\Delta+c)^{-1}g$ is the unique solution of $(\partial_1-a\Delta+c)h=g$ such that $h\in C_\mfs^{\beta+2}(\bbR\times\bbT^2)$ and $\lim_{x_1\to-\infty}h(x)=0$.
\item\cite[Theorem 63]{BHK22}
The operator $c-\mcL$ has an inverse of the form
$$
(c-\mcL)^{-1} f = \int_0^\infty Q_tf\,dt
=\int_0^1Q_tf\, dt+Q_1(c-\mcL)^{-1}f.
$$
For any $\zeta\in(-4,0)\setminus\bbZ$, the map $(c-\mcL)^{-1}$ uniquely extends to a continuous operator from $C^{\zeta,Q}(\bbR\times\bbT^2)$ to $C_\mfs^{\zeta+4}(\bbR\times\bbT^2)$.
\item\label{5:proposition:convolutionops3}
\cite[Theorem 6]{BHK22}
We can decompose $(\partial_1 - a\Delta+c)^{-1}=K+S$, where
$$
K :=: \int_0^1K_t\, dt:= -\int_0^1 (\partial_1+\Delta)Q_t\, dt 
$$
and
$$
S := K_1(c-\mcL)^{-1} + c(\partial_1 - a\Delta+c)^{-1}(1+\partial_1+\Delta) (c-\mcL)^{-1}.
$$
Then $\{K_t\}_{t>0}$ is a $2$-regularizing kernel admissible for $\{Q_t\}_{t>0}$ in the sense of Definition \ref{asmp2}, where $\delta\in(2,2+\alpha)$ in the condition \ref{asmp2:holder}.
Moreover, for any $\zeta\in(-2,0)\setminus\{-1\}$ and $\varepsilon>0$, $S$ is continuous from $C^{\zeta,Q}(\bbR\times\bbT^2)$ to $C_\mfs^{\alpha\wedge(\zeta+2)+2-\varepsilon}(\bbR\times\bbT^2)$.
\end{enumerate}
\end{prop}

\begin{rem}\label{rem3}
One needs to pick a constant $c>0$ large enough to construct the inverse operator $(c-\mcL)^{-1}$, see the proof of Theorem 63 in \cite{BHK22}. 
However, in the equation \eqref{5:eq:PAM}, $c$ can be an arbitrary constant. This is because we can replace the $c$ on the left-hand side with a larger constant $c'$ by adding a linear correction term $(c'-c)u(x)$ to the right-hand side.
This correction term has no serious influences on the discussion in this section. 
\end{rem}

%%%%%%%%%%%%%%%%%%%%%%%%%%%%%%%%%%%%
\subsection{Regularity structure associated with PAM}
%%%%%%%%%%%%%%%%%%%%%%%%%%%%%%%%%%%%

Following \cite{Hai14}, we prepare the regularity structure associated with PAM \eqref{5:eq:PAM}.

\begin{defi}\label{5:def:RSPAM}
For any fixed $\varepsilon\in(0,1/2)$, we define the regularity structure $\scT=(\bfA,\bfT,\bfG)$ of regularity $\alpha_0:=-1-\varepsilon$ as follows.
\begin{itemize}
\item[(1)]
(Index set) $\bfA=\{-1-\varepsilon,\, -2\varepsilon,\, -\varepsilon,\, 0,\, 1-\varepsilon,\, 1,\, 2-2\varepsilon,\, 2-\varepsilon\}$.
\item[(2)]
(Model space) $\bfT$ is an eleven dimensional linear space spanned by the symbols
$$
\Xi,\ \mcI(\Xi)\Xi,\ X_2\Xi,\ X_3\Xi,\ {\bf1},\ \mcI(\Xi),\ X_2,\ X_3,\ \mcI(\mcI(\Xi)\Xi),\ \mcI(X_2\Xi),\ \mcI(X_3\Xi).
$$
The direct sum decomposition $\bfT=\bigoplus_{\alpha\in \bfA}\bfT_\alpha$ is given by
\begin{align*}
\bfT_{-1-\varepsilon}&=\spa\{\Xi\},
&\bfT_{-2\varepsilon}&=\spa\{\mcI(\Xi)\Xi\},\\
\bfT_{-\varepsilon}&=\spa\{X_i\Xi\}_{i\in\{2,3\}},
&\bfT_0&=\spa\{{\bf1}\},\\
\bfT_{1-\varepsilon}&=\spa\{\mcI(\Xi)\},
&\bfT_1&=\spa\{X_i\}_{i\in\{2,3\}},\\
\bfT_{2-2\varepsilon}&=\spa\{\mcI(\mcI(\Xi)\Xi)\},
&\bfT_{2-\varepsilon}&=\spa\{\mcI(X_i\Xi)\}_{i\in\{2,3\}}.
\end{align*}
\item[(3)]
(Structure group) 
$\bfG$ is a group of continuous linear operators on $\bfT$ such that, for any $\Gamma\in \bfG$ and $\alpha\in\bfA$,
$$
(\Gamma-\id)\bfT_\alpha\subset\bfT_{<\alpha}.
$$
%$\bfG$ consists of all linear functionals on the linear space spanned by the symbols
%\begin{gather*}
%\mcI(\Xi),\ X_2,\ X_3,\ \mcI(\mcI(\Xi)\Xi),\ \mcI(X_2\Xi),\ \mcI(X_3\Xi),\\
%\mcI_2(X_2\Xi),\ \mcI_2(X_3\Xi),\ \mcI_3(X_2\Xi),\ \mcI_3(X_3\Xi).
%\end{gather*}
%The action of $g\in\bfG$ to $\bfT$ is defined by
%\begin{align*}
%\Gamma_g{\bf1}&={\bf1},
%&\Gamma_gX_i&=X_i+g(X_i){\bf1},
%&\Gamma_gX_i\Xi&=X_i\Xi+g(X_i)\Xi,\\
%\Gamma_g\Xi&=\Xi,
%&\Gamma_g\mcI(\Xi)&=\mcI(\Xi)+g(\mcI(\Xi)){\bf1},
%&\Gamma_g\mcI(\Xi)\Xi&=\mcI(\Xi)\Xi+g(\mcI(\Xi))\Xi,
%\end{align*}
%and
%\begin{align*}
%\Gamma_g\mcI(\tau\Xi)&=\mcI(\tau\Xi)+g(\mcI(\tau\Xi)){\bf1}
%+g(\tau)\mcI(\Xi)+\sum_{i\in\{2,3\}}g(\mcI_i(\tau\Xi))X_i
%\end{align*}
%for $\tau\in\{\mcI(\Xi),X_2,X_3\}$.
%The group structure $*:\bfG\times\bfG\to\bfG$ is determined by the relation $\Gamma_g\Gamma_h=\Gamma_{g*h}$.
\end{itemize}
\end{defi}

Although the above structure group is a more generic one copied from Definition \ref{*def:ATG} than the more particular one defined in \cite[Section 8]{Hai14}, we use the above definition to avoid preparing algebraic matters such as Hopf algebras and comodules. The above one is sufficient for the discussion in this section. The admissible model defined later is also realized in the particular structure group defined in \cite[Section 8]{Hai14}.
In what follows, let $\scT$ be the regularity structure given in Definition \ref{5:def:RSPAM} with fixed $\varepsilon$.

We consider the models and modelled distributions as in Section \ref{section:singular} with slight modifications.
For any $r\ge0$, we define the weight function
$$
v_r(x)=e^{-r|x_1|}.
$$
It is easy to see that $v_r$ satisfies the inequality \eqref{weightmoderate} with $v_r^*(x):=e^{r|x_1|}$ and $v_r$ is $G$-controlled.
Moreover, $v_r$ satisfies the assumption of Remark \ref{proposition:whatisPixtau} with $w_1 (x)= e^{-2r\|x\|}$ and $w_2 (x)= e^{-3r\|x\|}$, where $\|x\| := \sum_{i=1}^3 |x_i|$.

\begin{defi}\label{def:smoothadmissible}
We say that a smooth model $M\in\scM_{v_r}(\scT)$ (defined on $\bbR^3$) is \emph{admissible} if it satisfies the following properties.
\begin{enumerate}
\renewcommand{\theenumi}{(\roman{enumi})}
\renewcommand{\labelenumi}{(\roman{enumi})}
\item
For any $x,y\in\bbR^3$ and $i\in\{2,3\}$, we have
$$
\big(\Pi_{x+e_i}(\cdot)\big)(y+e_i)=\big(\Pi_x(\cdot)\big)(y),\qquad
\Gamma_{(y+e_i)(x+e_i)}=\Gamma_{yx}.
$$ 
%where $e_2=(0,1,0)$ and $e_3=(0,0,1)$.
\item
We write $\Pi\Xi=\Pi_x\Xi$ since it is independent of $x$.
For any $x\in\bbR^3$, we have
\begin{align*}
\Pi_x{\bf1}=1,\qquad
\Pi_xX_i=(\cdot)_i-x_i,\qquad
\Pi_x\mcI(\Xi)=K(\cdot,\Pi\Xi)-K(x,\Pi\Xi),
\end{align*}
and
\begin{align*}
\Pi_x\mcI(\tau\Xi)=K(\cdot,\Pi_x\tau\Xi)-K(x,\Pi_x\tau\Xi)-\sum_{i\in\{2,3\}}((\cdot)_i-x_i)\partial_iK(x,\Pi_x\tau\Xi),
\end{align*}
where $\tau\in\{\mcI(\Xi),X_2,X_3\}$.
\item
For any $x,y\in\bbR^3$, we have
\begin{align*}
\Gamma_{yx}{\bf1}&={\bf1},
&\Gamma_{yx}X_i&=X_i+(y_i-x_i){\bf1},\\
\Gamma_{yx}\Xi&=\Xi,
&\Gamma_{yx}\mcI(\Xi)&=\mcI(\Xi)+\big(K(y,\Pi\Xi)-K(x,\Pi\Xi)\big){\bf1},
\end{align*}
and
\begin{align*}
& \Gamma_{yx}(\tau\Xi)=\tau\Xi+(\Pi_x\tau)(y)\Xi,\\
& \Gamma_{yx}\mcI(\tau\Xi)=\mcI(\tau\Xi)+(\Pi_x\tau)(y)\mcI(\Xi)\\
&\quad+\bigg(K(y,\Pi_x\tau\Xi)-K(x,\Pi_x\tau\Xi)-\sum_{i\in\{2,3\}}(y_i-x_i)\partial_iK(x,\Pi_x\tau\Xi)\bigg){\bf1}\\
&\quad+\sum_{i\in\{2,3\}}\big(\partial_iK(y,\Pi_y\tau\Xi)-\partial_iK(x,\Pi_x\tau\Xi)\big)X_i,
\end{align*}
where $\tau\in\{\mcI(\Xi),X_2,X_3\}$.
\item
For any $\tau\in\{\Xi,\mcI(\Xi)\Xi,X_2\Xi,X_3\Xi,{\bf1}\}$, we have
$$
\sup_{x\in\bbR^d}v_r(x)|(\Pi_x\tau)(x)|<\infty.
$$
\end{enumerate}
We define the closed subspace $\scM_r^\per(\scT)$ of $\scM_{v_r}(\scT)$ as the completion of the set of smooth admissible models. 
\end{defi}

By definition, the subspace 
$$\bfS:=\spa\{{\bf 1},\mcI(\Xi),X_2,X_3,\mcI(\mcI(\Xi)\Xi),\mcI(X_2\Xi),\mcI(X_3\Xi)\}$$
is invariant under the action of admissible models.
In the sense of Definition \ref{defTbarT}, the linear operator $\mcI:\bfT\to\bfS$ defined by
$$
\mcI\tau=\begin{cases}
\mcI\tau&(\tau\in\{\Xi,\mcI(\Xi),X_2\Xi,X_3\Xi\})\\
0&(\tau\in\{{\bf1},\mcI(\Xi),X_2,X_3,\mcI(\mcI(\Xi)\Xi),\mcI(X_2\Xi),\mcI(X_3\Xi)\})
\end{cases}
$$
is an abstract integration of order $2$, and for any $M\in\scM_r^\per(\scT)$, the pair $(M,M)$ is $K$-admissible for $\mcI$.
Therefore, we can define the operator $\mcK$ by Corollary \ref{corollary:singularMSE}.

The weight function $v_r$ is used only to ensure the global bound of the model $M$ defined from the white noise.
For the definition of singular modelled distributions, the flat weight $v_0=1$ is sufficient since we study the local-in-time solution theory of \eqref{5:eq:PAM}.

\begin{defi}
For any interval $I\subset\bbR$ and any $\eta\le\gamma$, we define $\mcD^{\gamma,\eta}(I;\Gamma)$ as the space of all functions $f:(I\setminus\{0\})\times \bbT^2\to\bfT_{<\gamma}$ such that
\begin{align*}
\lp f\rp_{\gamma,\eta;I}&:=\max_{\alpha<\gamma}\sup_{x\in(I\setminus\{0\})\times\bbT^2}
\frac{\|f(x)\|_\alpha}{\varphi (x)^{(\eta-\alpha)\wedge0}}<\infty,\\
\| f\|_{\gamma,\eta;I}
&:=\max_{\alpha<\gamma}\sup_{\substack{x,y\in(I\setminus\{0\})\times\bbT^2,\,x\neq y \\ \|y-x\|_\mfs\le\varphi (x,y)}}
\frac{\|\Delta^\Gamma_{yx}f\|_\alpha}{\varphi (x,y)^{\eta-\gamma}\|y-x\|_\mfs^{\gamma-\alpha}}
<\infty.
\end{align*}
We denote by $\mcD^{\gamma,\eta}(I,\bfS;\Gamma)$ the subspace of $\bfS$-valued functions in the class $\mcD^{\gamma,\eta}(I;\Gamma)$.
\end{defi}

%%%%%%%%%%%%%%%%%%%%%%%%%%%%%%%%%%%%
\subsection{Convolution operators}\label{5:sec:KS}
%%%%%%%%%%%%%%%%%%%%%%%%%%%%%%%%%%%%

We can rewrite the equation \eqref{5:eq:PAM} in the form
\begin{equation}\label{5:eq:PAM'}
u(x)=\int_{\bbR^2}P_{x_1}(x',y')u_0(y')dy'
+(\partial_1-a\Delta+c)^{-1}\big\{{\bf1}_{(0,\infty)\times\bbR^2}b(u)\xi\big\}(x),
\end{equation}
where $u_0$ is the initial value of $u$ at $x_1=0$.
In this subsection, we prepare some operators to reformulate the equation \eqref{5:eq:PAM'} at the level of singular modelled distributions.

First, the function $Pu_0(x):=\int_{\bbR^2}P_{x_1}(x',y')u_0(y')dy'$ can be lifted to the singular modelled distribution taking values in the polynomial structure.
For any sufficiently regular function $f$ on $(\bbR\setminus\{0\})\times\bbR^2$, we define the $\bfT$-valued function
$$
Lf(x):=f(x){\bf1}+(\partial_2f)(x)X_2+(\partial_3f)(x)X_3
\qquad(x\in(\bbR\setminus\{0\})\times\bbR^2).
$$

\begin{lem}[{\cite[Lemma 29]{BHK22}}]\label{5:lemma:initialcondition}
Let $\theta\in(0,1)$ and $u_0\in C^{\theta}(\bbT^2)$.
Then the lift $L(Pu_0)$ of the function ${\bf1}_{x_1>0}Pu_0(x)$ is in the class $\mcD^{\gamma,\theta}$ for any $\gamma\in(0,2)$ and we have
$$
\|L(Pu_0)\|_{\gamma,\theta;(0,t)}
\lesssim\|u_0\|_{C^{\theta}(\bbT^2)}
$$
for any $t>0$.
\end{lem}

Next, to lift the second term on the right hand side of \eqref{5:eq:PAM'}, we prepare two lemmas.
The first one is used to ``extend" the domain of singular modelled distributions from $(0,t)\times\bbT^2$ to $\bbR\times\bbT^2$.

\begin{lem}\label{proposition:extensionMD}
We fix a smooth non-increasing function $\chi:(0,\infty)\to[0,1]$ such that
$$
\chi(t)=\begin{cases}
1&(0<t\le1),\\
0&(t\ge2).
\end{cases}
$$
For each $t>0$, we define the function $\chi_t:\bbR^3\to\bbR$ by setting $\chi_t(x)={\bf1}_{x_1>0}\chi(x_1/t)$.
Let $M=(\Pi,\Gamma)\in\scM_r^\per(\scT)$ with some $r>0$ and let $\gamma\in(0,1-2\varepsilon)$ and $\eta\le\gamma$.
For any $t\in(0,1]$ and any $f\in\mcD^{\gamma,\eta}((0,2t);\Gamma)$, we define the function
\begin{align*}
(E_tf)(x)=
P_{<\gamma}\big((L\chi_t)(x)\cdot f(x)\big),
\end{align*}
where the (partial) product $(\cdot)$ on $\bfT$ is defined by
\begin{align*}
{\bf1}\cdot\tau=\tau\quad(\tau\in\{\Xi,\mcI(\Xi)\Xi,X_2\Xi,X_3\Xi,{\bf1}\}),\qquad
X_i\cdot\Xi=X_i\Xi\quad(i\in\{2,3\}).
\end{align*}
(Other products do not appear due to the assumption on $\gamma$.)
Then the function $E_tf$ belongs to $\mcD^{\gamma,\eta\wedge\alpha_0}(\bbR;\Gamma)$ and satisfies
$$
\tri E_tf\tri_{\gamma,\eta\wedge\alpha_0;\bbR}\le C(1+\|\Gamma\|_{\gamma,v_r})\tri f\tri_{\gamma,\eta;(0,2t)}
$$
for some constant $C>0$ independent of $t$.
Moreover, $(E_tf)\vert_{(0,t]\times\bbT^2}= f\vert_{(0,t]\times\bbT^2}$.
\end{lem}

\begin{proof}
We can check that $\tri L\chi_t\tri_{\gamma',0;\bbR}\lesssim1$ for any $\gamma'\in(1,2)$  by definition, so by applying the continuity of the multiplication of modelled distributions \cite[Proposition 6.12]{Hai14}, we have
$$
\tri E_tf\tri_{\gamma,\eta\wedge\alpha_0;(0,2t)}\lesssim\tri f\tri_{\gamma,\eta;(0,2t)}.
$$
We can extend it into $\tri E_tf\tri_{\gamma,\eta\wedge\alpha_0;(0,2t]}\lesssim\tri f\tri_{\gamma,\eta;(0,2t)}$ by the uniform continuity.
To show that $E_tf\in\mcD^{\gamma,\eta\wedge\alpha_0}((0,\infty);\bbR)$, we pick $x\in[2t,\infty)\times\bbT^2$ and $y\in(0,2t)\times\bbT^2$.
By setting $z=(2t,y')$ we have
\begin{align*}
&\|(E_tf)(y)-\Gamma_{yx}(E_tf)(x)\|_\alpha\\
&\le\|(E_tf)(y)-\Gamma_{yz}(E_tf)(z)\|_\alpha+\|\Gamma_{yz}(E_tf)(z)-\Gamma_{yx}(E_tf)(x)\|_\alpha\\
&\le \|E_tf\|_{\gamma,\eta\wedge\alpha_0;(0,2t]}\,\varphi (y)^{\eta\wedge\alpha_0-\gamma}\|y-z\|_\mfs^{\gamma-\alpha}\\
&\lesssim\tri f\tri_{\gamma,\eta;(0,2t)}\,\varphi (x,y)^{\eta\wedge\alpha_0-\gamma}\|y-x\|_\mfs^{\gamma-\alpha}.
\end{align*}
In the second inequality, we use the fact that $(E_tf)(z)=(E_tf)(x)=0$ because of the definition of $E_t$.
For the case that $x\in(0,2t)\times\bbT^2$ and $y\in[2t,\infty)\times\bbT^2$, by the properties of models we have
\begin{align*}
&v_r(x)\|(E_tf)(y)-\Gamma_{yx}(E_tf)(x)\|_\alpha
=v_r(x)\|\Gamma_{yx}\{\Gamma_{xy}(E_tf)(y)-(E_tf)(x)\}\|_\alpha\\
&\le\|\Gamma\|_{\gamma,v_r}v_r^*(y-x)\sum_{\alpha\le\beta<\gamma}\|y-x\|_\mfs^{\beta-\alpha}\|\Gamma_{xy}(E_tf)(y)-(E_tf)(x)\|_\beta
\\
&\lesssim\|\Gamma\|_{\gamma,v_r}\tri f\tri_{\gamma,\eta;(0,2t)}\,v_r^*(y-x)
\varphi (x,y)^{\eta\wedge\alpha_0-\gamma}\|y-x\|_\mfs^{\gamma-\alpha}.
\end{align*}
Note that the supremum in the definition of the norm $\|\cdot\|_{\gamma,\eta;I}$ is taken over $\|y-x\|_\mfs\le\varphi (x,y)$. Since $|y_1|\le1+|x_1|\le3$ in this region, the factors $v_r(x)$ and $v_r^*(y-x)$ are bounded both above and below.
Thus we can ignore these weights and have $E_tf\in\mcD^{\gamma,\eta\wedge\alpha_0}((0,\infty);\Gamma)$.
On the other hand, $E_tf\in\mcD^{\gamma,\eta\wedge\alpha_0}((-\infty,0);\Gamma)$ is obvious from the definition.
Since $\|y-x\|_\mfs\le\varphi (x,y)$ implies that $x_1$ and $y_1$ have the same sign, we obtain the assertion.
\end{proof}

\begin{rem}
Although the norm of $\Pi$-parts of models is perhaps different from the original one in \cite{Hai14}, the norms of $\Gamma$-part and modelled distributions are not different since the semigroup $\{Q_t\}$ is not used for them. Because of this, here and in some places below (Lemma \ref{def:variantsnorms2} and Theorem \ref{5:eq:MDPAM}), we can use the continuity results of modelled distribution obtained in \cite{Hai14}. 
\end{rem}

Next, we recall from \cite{Hai14} a different norm of singular modelled distributions.
The following result holds for any singular modelled distributions on $\bbR^d$ taking values in arbitrary regularity structures and any models.

\begin{lem}[{\cite[lemma 6.5]{Hai14}}]\label{def:variantsnorms2}
Let $\eta\le\gamma$ and $r\ge0$, and let $I\subset\bbR$ be an interval.
For any function $f:(I\setminus\{0\})\times\bbT^2\to\bfT_{<\gamma}$, we define
$$
\lp f\rp_{\gamma,\eta;I}^\circ:=\max_{\alpha<\gamma}\sup_{x\in(I\setminus\{0\})\times\bbT^2}
\frac{\|f(x)\|_\alpha}{\varphi (x)^{\eta-\alpha}}.
$$
Then the inequality $\lp f\rp_{\gamma,\eta;I}\le\lp f\rp_{\gamma,\eta;I}^\circ$ obviously holds.
Conversely, if 
\begin{equation*}%\label{asmp:def:variantsnorms2}
\lim_{x_1\to0}P_\alpha f(x)=0
\end{equation*}
holds for any $\alpha<\eta$, then there exists a polynomial $p(\cdot)$ such that, for any $M\in\scM_r^\per(\scT)$ and $f\in\mcD^{\gamma,\eta}(I;\Gamma)$, we have
$$
\lp f\rp_{\gamma,\eta;I}^\circ\lesssim p(\|\Gamma\|_{\gamma,v_r})\tri f\tri_{\gamma,\eta;I}.
$$
\end{lem}

In the end, we can lift the operator $(\partial_1-a\Delta+c)^{-1}$ to the level of singular modelled distributions.
Recall the decomposition $(\partial_1-a\Delta+c)^{-1}=K+S$ from Proposition \ref{5:proposition:convolutionops}-\ref{5:proposition:convolutionops3}.

\begin{thm}\label{theorem:localintimeMSE}
Let $\gamma\in(0,\alpha\wedge(1-2\varepsilon))$, $\eta\in(\gamma-2,\gamma]$, $r\ge0$, and $t\in(0,1]$.
For any $M=(\Pi,\Gamma)\in\scM_r^\per(\scT)$, $f\in\mcD^{\gamma,\eta}((0,2t);\Gamma)$, and $\delta\in(0,\gamma+2]$, we define the function
$$
\mcP_t^\delta f:=P_{<\delta}\{\mcK(E_tf)+L(S(\mcR E_t f))\}.
$$
Then $\mcP_t^\delta f\in\mcD_{v_{3r}}^{\delta,\eta\wedge\alpha_0+2}(\bbR;\Gamma)$.
If $M$ is smooth and admissible in the sense of Definition \ref{def:smoothadmissible}, then we have
\begin{equation}\label{theorem:localintimeMSE:eq1}
\mcR(\mcP_t^\delta f)(x)=(\partial_1-a\Delta+c)^{-1}(\mcR E_tf)(x).
\end{equation}
Moreover, there exists a polynomial $p(\cdot)$ such that, for any $\kappa\ge0$ we have
\begin{equation}\label{theorem:localintimeMSE:eq2}
\tri \mcP_t^\delta f \tri_{\delta,\eta\wedge\alpha_0+2-\kappa;(0,2t)}
\le p(\tri M\tri_{\gamma,v_r})\, t^{\kappa/2}\tri f\tri_{\gamma,\eta;(0,2t)}.
\end{equation}
Finally, there exists a polynomial $q(\cdot)$ such that
\begin{align*}
& \tri \mcP_t^\delta f^{(1)};\mcP_t^\delta f^{(2)} \tri_{\delta,\eta\wedge\alpha_0+2-\kappa;(0,2t)}\\
&\le q(R)\, t^{\kappa/2}
\big(\tri M^{(1)};M^{(2)}\tri_{\gamma,v_r}+\tri f^{(1)};f^{(2)}\tri_{\gamma,\eta;(0,2t)}\big)
\end{align*}
for any $M^{(i)}\in\scM_r^\per(\scT)$ and $f^{(i)}\in\mcD^{\gamma,\eta}((0,2t);\Gamma^{(i)})$ with $i\in\{1,2\}$ such that $\tri M^{(i)}\tri_{\gamma,v_r}\le R$ and $\tri f^{(i)}\tri_{\gamma,\eta;(0,2t)}\le R$.
\end{thm}

\begin{proof}
In the proof of inequalities, due to the density argument, we can assume that the model $M$ is smooth.

We know $\mcK E_tf\in\mcD_{v_{3r}}^{\gamma+2,\eta\wedge\alpha_0+2}(\bbR;\Gamma)$ from Corollary \ref{corollary:singularMSE}, and $\mcR E_tf\in C^{\eta\wedge\alpha_0,Q}(v_{2r})$ from Corollary \ref{corollary:singularreconstruction}.
Moreover, since $E_tf(x)$ vanishes outside $[0,2]\times\bbT^2$, we also obtain $\mcR E_tf\in C^{\eta\wedge\alpha_0,Q}(\bbR\times\bbT^2)$ by modifying the proof of Theorem \ref{theorem:singularreconstruction}.
Then by Proposition \ref{5:proposition:convolutionops}-\ref{5:proposition:convolutionops3}, we have $S(\mcR E_tf)\in C_\mfs^{\gamma+2}(\bbR\times\bbT^2)$ and thus $L(S(\mcR E_t f)) \in \mcD^{\gamma+2,\gamma+2}(\bbR;\Gamma)$.
Therefore, $\mcP_t^\delta f\in\mcD_{v_{3r}}^{\delta,\eta\wedge\alpha_0+2}(\Gamma)$ by Proposition \ref{def:variantsnorms}-\ref{def:variantsnorms0}. 
The identity \eqref{theorem:localintimeMSE:eq1} follows from Theorem \ref{theorem:KR=RK} and the definition of $L(S(\mcR E_t f))$.

Note that $\tri\mcP_t^\delta f\tri_{\delta,\eta\wedge\alpha_0+2;(0,2t)}\le C_r\tri\mcP_t^\delta f\tri_{\delta,\eta\wedge\alpha_0+2,v_{3r}}$ for some $r$-dependent constant $C_r$.
We show \eqref{theorem:localintimeMSE:eq2} for $\kappa>0$ by applying Lemma \ref{def:variantsnorms2}.
By definition, the only index $\alpha\in\bfA$ of elements in $\bfS$ smaller than $\eta\wedge\alpha_0+2$ $(\le1-\varepsilon)$ is $\alpha=0$. 
Since $M$ is smooth, by Proposition \ref{proposition:reconstsmooth}, the $\bfT_0$-component of $\mcP_t^\delta f(x)$ is equal to
$$
\big(\Pi_x(\mcP_t^\delta f)(x)\big)(x)
=(\mcR\mcP_t^\delta f)(x)
=(\partial_1-a\Delta+c)^{-1}(\mcR E_tf)(x).
$$
Since $(\mcR E_tf)(y)=\big(\Pi_y(E_tf)(y)\big)(y)=0$ vanishes on $y\in(-\infty,0)\times\bbT^2$, we also have 
$$
(\partial_1-a\Delta+c)^{-1}(\mcR E_tf)(x)
=\int_{[0,x_1]\times\bbR^2}P_{x_1-y_1}(x',y')(\mcR E_tf)(y)dy.
$$
Note that, in the proof of Proposition \ref{proposition:reconstsmooth}, we obtained
$$
|\mcR E_tf(y)|\lesssim\varphi (y)^{\eta\wedge\alpha_0}.
$$
Since $\eta\wedge\alpha_0>-2$, we can show that 
$$
|(\partial_1-a\Delta+c)^{-1}(\mcR E_tf)(x)|\lesssim\int_0^{x_1}|y_1|^{(\eta\wedge\alpha_0)/2}dy_1\to0
$$
as $x_1\downarrow0$.
Therefore, by Lemma \ref{def:variantsnorms2} we have
\begin{align*}
\tri \mcP_t^\delta f\tri_{\gamma,\eta\wedge\alpha_0+2-\kappa;(0,2t)}
&\lesssim\tri \mcP_t^\delta f\tri_{\gamma,\eta\wedge\alpha_0+2-\kappa;(0,2t)}^\circ\\
&\lesssim t^{\kappa/2}\tri \mcP_t^\delta f\tri_{\gamma,\eta\wedge\alpha_0+2;(0,2t)}^\circ
\lesssim t^{\kappa/2}\tri \mcP_t^\delta f\tri_{\gamma,\eta\wedge\alpha_0+2;(0,2t)},
\end{align*}
where $\tri \cdot\tri_{\gamma,\eta;I}^\circ:=\lp \cdot\rp_{\gamma,\eta;I}^\circ+\|\cdot\|_{\gamma,\eta;I}$.
The proof of the local Lipschitz estimate is a slight modification.
\end{proof}

%%%%%%%%%%%%%%%%%%%%%%%%%%%%%%%%%%%%
\subsection{Solution theory for PAM}
%%%%%%%%%%%%%%%%%%%%%%%%%%%%%%%%%%%%

We show the local-in-time well-posedness of the equation
\begin{equation}\label{5:eq:MDPAM}
U=L(Pu_0)+\mcP_t^\gamma \big(b(U)\Xi\big)
\end{equation}
in the class $\mcD^{\gamma,\eta}((0,2t),\bfS;\Gamma)$ with some appropriate choices of $\gamma$ and $\eta$.
The term $L(Pu_0)$ and the operator $\mcP_t^\gamma$ was defined in the previous subsection.
The only undefined object $b(U)$ is the lift of the composition map $u\mapsto b(u)$ defined in \cite[Proposition 6.13]{Hai14}.
In the present case, for sufficiently small $\varepsilon$ and any $U\in \mcD^{\gamma,\eta}((0,2t),\bfS;\Gamma)$ with $\gamma\in(1,2-2\varepsilon)$ and $\eta\in[0,\gamma]$ of the form
$$
U(x)=u(x){\bf1}+v(x)\mcI(\Xi)+u_2(x)X_2+u_3(x)X_3,
$$
we can define $b(U)\in\mcD^{\gamma,\eta}((0,2t),\bfS;\Gamma)$ by the concrete form
$$
b(U)(x)=b(u(x)){\bf1}+b'(u(x))\{v(x)\mcI(\Xi)+u_2(x)X_2+u_3(x)X_3\}.
$$
Then the map $U\mapsto b(U)$ is locally Lipschitz continuous.

\begin{thm}\label{5:theorem:FP}
Assume $\varepsilon\in(0,\alpha\wedge(1/4))$ and let $\theta\in(0,1-\varepsilon)$.
Then there exists a function $t_0:(0,\infty)^2\to(0,1]$ such that, the following assertion holds for any $R_1,R_2>0$:
For any $u_0\in C^\theta(\bbT^2)$ such that $\|u_0\|_{C^\theta(\bbT^2)}\le R_1$, and any $M\in\scM_r^\per(\scT)$ such that $\|M\|_{\gamma,v_r}\le R_2$, the equation \eqref{5:eq:MDPAM} with $t=t_0(R_1,R_2)$ and $\gamma=1+2\varepsilon$ has a unique solution $U$ in the class $\mcD^{1+2\varepsilon,\theta}((0,2t),\bfS;\Gamma)$.
Moreover, the mapping 
$$
S_t:(u_0,M)\mapsto U
$$
is Lipschitz continuous on the space $\{u_0\,;\,\|u_0\|_{C^\theta(\bbT^2)}\le R_1\}\times\{M\,;\, \|M\|_{\gamma,v_r}\le R_2\}$.
\end{thm}

\begin{proof}
The proof is a standard fixed point argument.
Note that, the following operators are well-defined and locally Lipschitz continuous.
\begin{itemize}
\item
(\cite[Proposition 6.13]{Hai14})
$U\in\mcD^{1+2\varepsilon,\theta}((0,2t),\bfS;\Gamma)\mapsto
b(U)\in\mcD^{1+2\varepsilon,\theta}((0,2t),\bfS;\Gamma)$.
\item
(\cite[Proposition 6.12]{Hai14})
$V\in\mcD^{1+2\varepsilon,\theta}((0,2t),\bfS;\Gamma)\mapsto
V\Xi\in\mcD^{\varepsilon,\theta-1-\varepsilon}((0,2t);\Gamma)$.
\item
(Theorem \ref{theorem:localintimeMSE})
$W\in\mcD^{\varepsilon,\theta-1-\varepsilon}((0,2t);\Gamma)\mapsto
\mcP_t^{1+2\varepsilon}W\in\mcD^{1+2\varepsilon,1-\varepsilon}\in((0,2t),\bfS;\Gamma)$.
\end{itemize}
Therefore, by setting $F(U)=L(Pu_0)+\mcP_t^{1+2\varepsilon} \big(b(U)\Xi\big)$, we have
\begin{align*}
\tri F(U)\tri_{1+2\varepsilon,\theta;(0,2t)}
&\lesssim\|u_0\|_{C^\theta}+t^{(1-\varepsilon-\theta)/2}\tri b(U)\Xi\tri_{\varepsilon,\theta-1-\varepsilon}\\
&\lesssim\|u_0\|_{C^\theta}+t^{(1-\varepsilon-\theta)/2}\tri b(U)\tri_{1+2\varepsilon,\theta}\\
&\lesssim\|u_0\|_{C^\theta}+t^{(1-\varepsilon-\theta)/2}p(\tri U\tri_{1+2\varepsilon,\theta})
\end{align*}
for some polynomial $p(\cdot)$.
From this inequality, we can find a large $R>0$ depending on $u_0$ and $M$ and show that $F$ maps a ball of radius $R$ in $\mcD^{1+2\varepsilon,\theta}((0,2t),\bfS;\Gamma)$ into itself.
From here onward, we can show the assertion by an argument similar to \cite[Theorem 7.8]{Hai14}.
\end{proof}

%%%%%%%%%%%%%%%%%%%%%%%%%%%%%%%%%%%%
\subsection{Convergence of models}
%%%%%%%%%%%%%%%%%%%%%%%%%%%%%%%%%%%%

In this subsection, we define the sequence of smooth admissible models associated with regularized noises and show its probabilistic convergence.
We fix an even function $\varrho:\bbR^2\to[0,1]$ in the Schwartz class and such that $\int_{\bbR^2}\varrho(x)dx=1$, and set $\varrho_n(x)=2^{2n}\varrho(2^{n}x)$ for each $n\in\bbN$.
We define the smooth approximation of the spatial white noise $\xi$ by
$$
\xi_n(x)=\int_{\bbT^2}\widetilde{\rho_n}(x-y)\xi(y)dy,\qquad(x\in\bbT^2)
$$
where $\widetilde{\rho_n}$ denotes the spatial periodization of $\rho_n$ defined by $\widetilde{\rho_n}(x):=\sum_{k\in\bbZ^2}\rho_n(x+k)$.
For such $\xi_n$, we can define the unique smooth admissible model $M^n=(\Pi^n,\Gamma^n)\in\scM_r^\per(\scT)$ by the properties
\begin{align*}
&(\Pi_x^n\Xi)(y)=\xi_n(y'),\qquad(\Pi_x^nX_i\Xi)(y)=(y_i-x_i)\xi_n(y'),\\
&\big(\Pi_x^n\mcI(\Xi)\Xi\big)(y)=\big(K\xi_n(y)-K\xi_n(x)\big)\xi_n(y')-C_n(y),
\end{align*}
where the function $C_n$ is defined by
$$
C_n(x)=\bbE\big[(K\xi_n)(x)\xi_n(x')\big]
=\int_{\bbR^3}K(x,y)c_n(x'-y')dy
$$
with $c_n(x'-y'):=\bbE[\xi_n(x')\xi_n(y')]=\widetilde{\varrho_n^{*2}}(x'-y')$.

\begin{thm}\label{5:theorem:modellimit}
For any $r>0$ and $p\in[1,\infty)$, the sequence $\{M^n\}_{n\in\bbN}$ of models defined above converges in $L^p(\Omega,\scM_r^\per(\scT))$.
\end{thm}

\begin{proof}
In view of the inductive proof as in \cite{BH23}, it is sufficient to show the uniform bounds
\begin{equation}\label{5:theorem:proof:modellimit}
\big|\bbE\big[Q_t(x,\Pi_x^n\tau)\big]\big|\lesssim t^{\beta/4}
\end{equation}
for any $\beta\in\{-1-\varepsilon,-2\varepsilon,-\varepsilon\}$ and $\tau\in\bfT_\beta$.
Note that the assumptions in \cite{BH23} are more restrictive: the kernel $Q_t(x,y)$ is homogeneous in the sense that it depends only on $x-y$, and the renormalization model is defined from an $x$-independent preparation map. However, the first restriction is used only to prove the above estimate in \cite{BH23}, so if we can establish this estimate in some alternative way, we can still follow the discussion in \cite{BH23}. Moreover, the second restriction is also not problematic, as the algebraic relations derived from preparation maps can be easily adapted to include $x$-dependent preparation maps. Such a modification is carried out in \cite{BHK22}.

Since $\xi$ is a centered Gaussian, we have only to show \eqref{5:theorem:proof:modellimit} for $\tau=\mcI(\Xi)\Xi$. By definition,
\begin{align*}
\bbE\big[Q_t(x,\Pi_x^n\tau)\big]
&=-\int_{\bbR^3}Q_t(x,y)\bbE[(K\xi_n)(x)\xi_n(y')]dy\\
&=-\int_{(\bbR^3)^2}Q_t(x,y)K(x,z)c_n(z'-y')dydz.
\end{align*}
To estimate this integral, we decompose $K=\int_0^1K_sds$ and set
$$
I_{t,s}^n(x)=-\int_{(\bbR^3)^2}Q_t(x,y)K_s(x,z)c_n(z'-y')dydz.
$$
By the Gaussian estimates of $Q_t$ and $K_s$, their time integral is estimated as
\begin{align*}
\int_{\bbR}|Q_t(x,y)|dy_1\lesssim h_t^{(C)}(x'-y'),\qquad
\int_{\bbR}|K_s(x,z)|dz_1\lesssim s^{-1/2}h_s^{(C)}(x'-z'),
\end{align*}
for some constant $C>0$, where $h_t^{(C)}(x'):=t^{-1/2}e^{-C\{(|x_2|^4/t)^{1/3}+(|x_3|^4/t)^{1/3}\}}$.
Thus we have
\begin{align*}
|I_{t,s}^n(x)|
\lesssim s^{-1/2}(h_t^{(C)}*h_s^{(C)}*|c_n|)(0).
\end{align*}
Since $|h_t^{(C)}*h_s^{(C)}(x)|\lesssim h_{t+s}^{(c)}(x)$ for some constant $c\in(0,C)$ (see \cite[Lemma 55]{BHK22} for instance), we have
$$
|I_{t,s}^n(x)|\lesssim s^{-1/2}(t+s)^{-1/2}.
$$
Since we have
\begin{align*}
\int_0^1|I_{t,s}^n(x)|ds
\lesssim\int_0^t s^{-1/2}t^{-1/2}ds+\int_t^1s^{-1}ds
\lesssim -\log t\lesssim t^{-\varepsilon/2}
\end{align*}
for any $\varepsilon>0$, we obtain the estimate \eqref{5:theorem:proof:modellimit} for $\tau=\mcI(\Xi)\Xi$.
\end{proof}

%%%%%%%%%%%%%%%%%%%%%%%%%%%%%%%%%%%%
\subsection{Renormalization of PAM}\label{5:sec:renorPAM}
%%%%%%%%%%%%%%%%%%%%%%%%%%%%%%%%%%%%

For a fixed initial condition $u_0\in C^\theta(\bbT^2)$ and the sequence of random models $\{M^n\}$ constructed in the previous subsection, we denote by
$$
U_n=S_t(u_0,M^n)
$$
the solution of the equation \eqref{5:eq:MDPAM} with $\gamma=1+2\varepsilon$ and with the random time
$$
t=t_0\bigg(\|u_0\|_{C^\theta(\bbT^2)},\ \sup_{n\in\bbN}\|M^n\|_{\gamma,v_r}\bigg).
$$
Combining Theorem \ref{5:theorem:modellimit} with Theorem \ref{5:theorem:FP}, we have the following theorem.

\begin{thm}\label{5:theorem:renorPAM}
For each $n\in\bbN$, we denote by $\mcR^n$ the reconstruction operator associated with $M^n$.
Then the function $u_n=\mcR^n(E_tU_n)$ converges in $L^\infty((0,t)\times\bbT^2)$ in probability as $n\to\infty$ and coincides with the unique solution of the equation
\begin{equation}\label{5:theorem:renorPAMeq}
\big(\partial_1-a(x')\Delta +c\big)u_n(x)=b\big(u_n(x)\big)\xi_n(x')-C_n(x)(bb')\big(u_n(x)\big)
\end{equation}
with the initial value $u_0\in C^\theta(\bbT^2)$ on $x\in(0,t)\times\bbT^2$.
\end{thm}
As noted in Remark \ref{rem3}, the constant $c$ in the equation \eqref{5:theorem:renorPAMeq} can be arbitrary.
\begin{proof}
On the region $x\in(0,t)\times\bbT^2$, since $u_n(x)=\big(\Pi_x^nU_n(x)\big)(x)$, we can assume that $U_n$ is of the form
\begin{equation}\label{5:theorem:renorPAMeqproof}
U_n(x)=u_n(x){\bf1}+v_n(x)\mcI(\Xi)+u_{2,n}(x)X_2+u_{3,n}(x)X_3.
\end{equation}
The convergence of $\{u_n\}$ in $L^\infty((0,t)\times\bbT^2)$ follows from the convergence of $\{U_n\}$ and the definition of the norm $\lp\cdot\rp_{\gamma,\eta;(0,t)}$.

Finally, we show that $u_n$ satisfies the equation \eqref{5:theorem:renorPAMeq} on the region $(0,t)\times\bbT^2$.
For any $x\in(0,t)\times\bbT^2$, the function $b(U_n)(x)$ is of the form
$$
b(U_n)(x)=b(u_n(x)){\bf1}+b'(u_n(x))\{v_n(x)\mcI(\Xi)+u_{2,n}(x)X_2+u_{3,n}(x)X_3\},
$$
and then $\mcP_t^{1+2\varepsilon}\big(b(U)\Xi\big)$ is of the form
$$
\mcP_t^{1+2\varepsilon}\big(b(U)\Xi\big)(x)
=w_n(x){\bf1}+b(u_n(x))\mcI(\Xi)+w_{2,n}(x)X_2+w_{3,n}(x)X_3
$$
for some functions $w_n,w_{2,n}$, and $w_{3,n}$.
For $U_n$ to solve the equation \eqref{5:eq:MDPAM}, the coefficient $v_n(x)$ in \eqref{5:theorem:renorPAMeqproof} must be equal to $b(u_n(x))$ for any $x\in(0,t)\times\bbT^2$.
By Theorem \ref{theorem:localintimeMSE}, the function $u_n$ satisfies
\begin{align*}
u_n(x)=Pu_0(x)+\int_{[0,x_1]\times\bbR^2}P_{x_1-y_1}(x',y')(\mcR^n E_tf(U_n)\Xi)(y)dy.
\end{align*}
Since $y\in(0,t)\times\bbT^2$, from the definition of $\Pi_x^n\mcI(\Xi)\Xi$, we obtain
\begin{align*}
(\mcR E_tb(U_n)\Xi)(y)
&=\big(\Pi_yE_tb(U_n)(y)\Xi\big)(y)
=\big(\Pi_yb(U_n)(y)\Xi\big)(y)\\
&=b(u_n(y))\xi_n(y')-C_n(y)(bb')(u_n(y)).
\end{align*}
This implies that $u_n$ satisfies the equation \eqref{5:theorem:renorPAMeq} (in mild sense) on $(0,t)\times\bbT^2$.
\end{proof}

We also have a stronger convergence result.

\begin{cor}
In the setting of Theorem \ref{5:theorem:renorPAM}, the convergence of $\{u_n\}$ also holds in the space $C_\mfs^\theta((0,t)\times\bbT^2)$.
\end{cor}

\begin{proof}
We only show the uniform bounds of $\{u_n\}$ in the $\theta$-H\"older norm, since the proof of the convergence is a simple modification. First, we set $\bar{U}_n=\mcP_t^{1+2\varepsilon}\big(b(U_n)\Xi\big)\in \mcD^{1+2\varepsilon,\theta}((0,2t),\bfS;\Gamma^n)$ and decompose 
$$
u_n=Pu_0+\bar{u}_n,\qquad\bar{u}_n:=\mcR^n(E_t\bar{U}_n).
$$
Since the uniform bounds of $\{Pu_0\}$ in the $\theta$-H\"older norm is more elementary (see e.g. \cite[Proposition 62]{BHK22}), we focus on the remaining term. By definition, for any $x\in(0,t)\times\bbT^2$, $\bar{u}_n(x)$ coincides with the ${\bf1}$-component of $\bar{U}_n(x)$, and also with that of $P_{<\theta}\bar{U}_n(x)$. Since $\{P_{<\theta}\bar{U}_n\}$ is uniformly bounded in the norm $\|\cdot\|_{\theta,\theta;(0,t)}$ by Proposition \ref{def:variantsnorms}-\ref{def:variantsnorms0}, we have
$$
|\bar{u}_n(y)-\bar{u}_n(x)|\lesssim\|y-x\|_\mfs^\theta
$$
for any $x,y\in(0,t)\times\bbT^2$ such that $\|y-x\|_\mfs\le\varphi(x,y)$. 
Here and in what follows, we omit proportional constants polynomially depending on the norms of $\{\bar{U}_n\}$ and $\{\Gamma^n\}$, which are uniform over $n$.
It remains to show the same H\"older-type inequality in the region $\varphi(x,y)<\|y-x\|_\mfs$. In this region, by using the inequality \eqref{ineq:omegashift}, we have $\varphi(x)\vee\varphi(y)\lesssim\|y-x\|_\mfs$.
On the other hand, we also have that $\{\bar{U}_n\}$ is uniformly bounded in the norm $\lp\cdot\rp_{1+2\varepsilon,\theta;(0,t)}^\circ$ by Lemma \ref{def:variantsnorms2}. Hence
\begin{align*}
|\bar{u}_n(y)-\bar{u}_n(x)|&\le|\bar{u}_n(y)|+|\bar{u}_n(x)|
\lesssim\varphi(y)^\theta+\varphi(x)^\theta\lesssim\|y-x\|_\mfs^\theta
\end{align*}
in the region $\varphi(x,y)<\|y-x\|_\mfs$. This completes the proof.
\end{proof}

\vspace{7mm}
\noindent
{\bf Acknowledgements.}
%\section*{Acknowledgement}
The first author is supported by JSPS KAKENHI Grant Number 23K12987.
The second author is supported by JST Grant Number JPMJSP2138.

\end{document}